\documentclass[12pt]{amsart}

\input xy
\usepackage[english]{babel}
\usepackage[latin1]{inputenc}
\usepackage{graphicx}
\usepackage{amsmath}
\usepackage{amssymb}
\usepackage{amscd}
\usepackage{color}
\usepackage{oldgerm}
\usepackage{amsfonts}
\usepackage{newlfont}
\usepackage{longtable}
\usepackage{multirow}
\usepackage{xcolor}

\DeclareOldFontCommand{\rm}{\normalfont\rmfamily}{\mathrm}

\usepackage[all]{xy}

\usepackage{upref}
\def\F{\Bbb F}

\def\N{\Bbb N}

\def\ad{\operatorname{ad}}

\def\Aut{\operatorname{Aut}}

\def\d{\operatorname{d}}

\def\det{\operatorname{det}}

\def\dim{\operatorname{dim}}

\def\End{\operatorname{End}}
\def\GL{\operatorname{GL}}

\def\Hom{\operatorname{Hom}}
\def\Ker{\operatorname{Ker}}

\def\Span{\operatorname{Span}}

\def\Im{\operatorname{Im}}

\def\D{\d}

\def\g{\frak g}
\def\gl{\frak{gl}}
\def\h{\frak h}

\def\a{\frak{a}}

\def\j{\frak{j}}

\def\ide{\frak{i}}

\theoremstyle{plain}\swapnumbers

\newtheorem{Theorem}{Theorem}[section]
\newtheorem{Lemma}[Theorem]{Lemma}
\newtheorem{Prop}[Theorem]{Proposition}

\newtheorem{Cor}[Theorem]{Corollary}
\newtheorem{Example}[Theorem]{Example}
\newtheorem{Remark}[Theorem]{Remark}

\newtheorem{claim}[Theorem]{Claim}

\title[On solvable quadratic Lie algebras]
{On solvable quadratic Lie algebras having an Abelian descending central ideal}

\author[Garc\'ia-Delgado et al]
{R. Garc\'{\i}a-Delgado$^{(a)}$, G. Salgado$^{(b)}$, O.A. S\'anchez-Valenzuela$^{(c)}$}
\address{(a)\&(c) Centro de Investigaci\'on en Matem\'aticas, A.C., 
Unidad M\'erida; Yucat\'an, M\'exico}
\address{(b) Facultad de Ciencias, Universidad Aut\'onoma de San Luis Potos\'{i}; 
San Luis Potos\'{\i}, M\'exico}
\address{(a)}
\email{rosendo.garcia@cimat.mx}
\address{(b)}
\email{gsalgado@fciencias.uaslp.mx, gil.salgado@gmail.com}
\address{(c)}
 \email{adolfo@cimat.mx}


\date{\today}

\begin{document}

\maketitle

\begin{abstract}
Solvable Lie algebras
having at least one Abelian descending central ideal are studied.
It is shown that all such Lie algebras can be built up from
canonically defined ideals. The nature of such ideals
is elucidated and their construction is provided in detail.
An approach to study and to classify these Lie algebras is given through
the theory of extensions via appropriate cocycles and representations
on which a group action is naturally defined.
Also, necessary and sufficient conditions for the existence of 
invariant metrics on the studied extensions are given.
It is shown that any solvable quadratic Lie algebra $\g$
having an Abelian descending central ideal
is of the form $\g=\h\oplus\a\oplus\h^*$, where 
$\h^*\simeq \ide(\g)$ and $\a \oplus \h^{*}\simeq \j(\g)$ are in fact
two canonically defined Abelian ideals of $\g$ satisfying $\ide(\g)^\perp=\j(\g)$.
As an example, a classification of this type of quadratic Lie algebras is given
assuming that $\j(\g)/\ide(\g)$ is an $r$-dimensional vector space
and $\g/\j(\g)$ is the 3-dimensional Heisenberg Lie algebra.
\end{abstract}

\section*{Introduction}

Let $\g$ be a finite-dimensional Lie algebra
over a field $\Bbb F$ of characteristic zero. 
The Lie algebra $\g$ is said to be {\it quadratic\/} if it comes equipped
with a non-degenerate, symmetric, bilinear form, $B:\g\times\g\to\Bbb F$, satisfying,
$B([x,y],z)=B(x,[y,z])$, for any $x$, $y$ and $z$ in $\g$; $[x,y]$ being the Lie bracket
of $x$ and $y$ in $\g$. The bilinear form $B$ is said to be {\it invariant\/} if this property is satisfied.

\smallskip
A theorem by A. Medina and P. Revoy \cite{Med} states that an
indecomposable, non-semisimple, quadratic Lie algebra $\g$
can be build up from a {\it minimal isotropic ideal\/}
$I$ and a quadratic Lie algebra which is isomorphic to $I^{\perp}/I$
and not necessarily indecomposable. 
In fact, such a quadratic Lie algebra $\g$ has a Witt vector space decomposition
with respect to the minimal isotropic ideal $I$.

\smallskip
Even though the Medina-Revoy Theorem deals in a clever and general
way with indecomposable, non-semisimple, quadratic Lie algebras, its proof depends 
on the choice of the minimal ideal $I$. This choice, however, is not unique. 
Besides, even though the subspace $I^{\perp}/I$ has the structure of a quadratic Lie algebra,
it is not in general a Lie subalgebra of $\g$. Moreover, $I^{\perp}/I$ might be further decomposed 
into mutually orthogonal subspaces, even if the algebra $\g$ one started with was indecomposable.

\smallskip
Later on, I. Kath and M. Olbrich proposed 
an alternative approach in \cite{Kath} looking for
a {\it canonical isotropic ideal\/} that might play the role of $I$;
they succeded in discovering a way to get a canonical ideal, though it
might not be minimal. On the other hand, an important asset of \cite{Kath} is 
that it brings to the foreground interesting ideas pointing toward the use of
cochain complexes and cohomology techniques based on elementary group actions
build up from morphisms and representations.

\smallskip
It was the work of Kath and Olbrich \cite{Kath} what made us realize 
that there might be another approach to
the Witt decomposition of $\g$ associated to a 
minimal ideal; namely,
by considering non-Abelian cuadratic Lie algebras $\g$ having
at least one Abelian descending central ideal; 
{\it ie\/,}
$\g^{\ell}:=[\g,\g^{\ell-1}]$ with $[\g^{\ell},\g^{\ell}]=\{0\}$, for some $\ell \in \N$.
The way of doing it is to use this hypothesis in order to produce
two {\it characteristic} ({\it ie\/,} canonically defined)
{\it Abelian ideals\/,} $\ide(\g)$ and $\j(\g)$, with $\ide(\g)\subset\j(\g)$
(see \S{2.2} below).
Furthermore, these ideals are such that $[\g,\j(\g)]\subset\ide(\g)$,
which makes the case $\ide(\g)\varsubsetneq \j(\g)$ particularly interesting. 
Thus, one obtains a Lie algebra $\h=\g/\j(\g)$, 
and a vector space $\a=\j(\g)/\ide(\g)$,
together with representations of $\h$ in $\j(\g)$ and $\ide(\h)$ that have 
special properties due to the fact that $[\g,\j(\g)]\subset\ide(\g)$. 

\smallskip
It is clear that any non-Abelian Lie algebra having an
Abelian descending central ideal is solvable. 
It is also clear that the class of all such Lie algebras includes the nilpotent ones.
At this point we remark that the Abelian ideals $\ide(\g)$ and $\j(\g)$ referred to above
are different from those studied in \cite{Kath} when the Lie algebra is non-nilpotent.
On the other hand, although much has already been said about solvable quadratic Lie algebras,
the structure theorem provided in this work ({\bf Thm. \ref{nilp cuad}} below)
gives a simple way to determine whether a solvable Lie algebra
of the class under study admits an invariant metric or not.
Besides, {\bf Thm. \ref{nilp cuad}} is particularly easy to state and prove
as it only makes use of classical tools.
Said tools include a naturally defined differential map in appropriatly defined
chain of complexes adapted to the $\h=\g/\j(\g)$ action
(see {\bf Prop. \ref{anticonmutativo}} and {\bf Rmrk. 2.2}).

\smallskip
To illustrate the usefulness of {\bf Thm. \ref{nilp cuad}},
we have worked out in full detail in \S4
the classification, up to isomorphism, of the Lie algebras of the form
$\g=\h\oplus\a\oplus\h^*$, 
for which $\h$ is the $3$-dimensional Heisenberg Lie algebra, 
$\ide(\g)$ is isomorphic to $\h^*$ acted on by the coadjoint representation,
and $\j(\g)/\ide(\g)=\a \simeq\Bbb F^r$, with $r \geq 3$.
It turns out that there are nine different families of such
isomorphism classes (see {\bf Prop. \ref{clasif de algs}}),
and only a finite set of specific representives
inside four of them admit an invariant metric
(see {\bf Prop. \ref{algs con metrica}}).

\smallskip
\section{Background on Abelian Extensions}

\subsection{Abelian Extensions of Lie algebras}

Let $\g$ be a Lie algebra with Lie bracket $[\,\cdot\,,\,\cdot\,]:\g\times\g\to\g$
and let $\j$ be an Abelian ideal of $\g$. 
Let $\h$ be a vector subspace complementary to $\j$, so that $\g=\h \oplus \j$.
For each pair $x$, $y \in \h$, let $[x,y]_{\h}$ and $\Lambda(x,y)$ be the
components of $[x,y]$ along $\h$ and $\j$, so that,
\begin{equation}\label{descomposicion corchete}
[x,y]=[x,y]_{\h}+\Lambda(x,y).
\end{equation}
One also obtains a representation $R$ of $\h$ in $\j$ via,
$$
\h\ni x\mapsto R(x)=\left(\ad\vert_{\h}(x)\right)\!\vert_{\j}=
[x,\,\cdot\,]\vert_{\j}\in\gl(\j).
$$
Indeed, being given by the adjoint representation, $R$ satisfies,
$$
\aligned
R([x,y])&=\ad([x,y])\vert_{\j} =
R(x)\circ R(y) - R(y)\circ R(x)
\\
& =\ad([x,y]_{\h}+\Lambda(x,y))\vert_{\j}.
\endaligned
$$
By restricting the adjoint action to the ideal $\j$, we have
$\left[\,\Lambda(x,y)\,,\,\cdot\,\right]\vert_{\j}\equiv 0$, since $\Lambda(x,y)\in\j$
and $\j$ is Abelian. Thus, for any pair $x$, $y \in \h$,
$$
\ad([x,y])\vert_{\j} = \ad([x,y]_{\h})\vert_{\j}=R([x,y]_{\h})
$$
and therefore, for any $x$ and $y$ in $\h$, we have,
\begin{equation}\label{representacion R}
R([x,y]_{\h})=R(x)\circ R(y) - R(y)\circ R(x).
\end{equation}
Now, for any three elements in $\h$, say $x$, $y$ and $z$, we have,
$$
\left[x,[y,z]\right]=\left[x,[y,z]_{\h}\right]_{\h}+\Lambda(x,[y,z]_{\h})+R(x)\left(\Lambda(x,y)\right).
$$
Since  the Lie bracket $[\,\cdot\,,\,\cdot\,]$ satisfies the Jacobi identity, it follows that,
\begin{equation}
\begin{array}{ll}
\displaystyle{\sum_{\circlearrowleft\{x,y,z\}}}\!\!\!\!\!\!&\left[x,[y,z]_{\h}\right]_{\h}=0,\qquad\text{and}
\\
\displaystyle{\sum_{\circlearrowleft\{x,y,z\}}}\!\!\!\!\!\!&\left(R(x)(\Lambda(y,z)) -\Lambda([y,z]_{\h},x)\right)=0.
\end{array}
\end{equation}
The first equation states that $\h=\g/\j$ is a Lie algebra under $[\,\cdot\,,\,\cdot\,]_{\h}$,
whereas the second equation states that $\Lambda$ is a $2$-cocycle in the
cochain complex $C(\h;\j)$ of alternating multilinear maps $\h\times\cdots\times\h\to\j$
into the $\h$-module $\j$ defined by the representation $R$. That is,
$$
(\d\Lambda)(x,y,z) = \!\!\!\sum_{\circlearrowleft\{x,y,z\}}\!\!\!\left\{\,R(x)(\Lambda(y,z)) -\Lambda([y,z]_{\h},x)\right\}=0.
$$
Conversely, let $\h$ be a Lie algebra with Lie bracket $[\,\cdot\,,\,\cdot\,]_{\h}$
and let $\Lambda:\h\times\h\to\j$ be a $2$-cocycle with values in 
the $\h$-module $\j$ defined by a given representation $R:\h\to\gl(\j)$.
It is well known that the skew-symmetric bilinear map $[\,\cdot\,,\,\cdot\,]:\g\times\g\to\g$,
defined on the direct sum $\g=\h\oplus\j$ by means of,
\begin{equation}\label{corchete de ext abeliana}
\aligned
\,[x,y] & = [x,y]_{\h} + \Lambda(x,y),
\\
\,[x,v] & = R(x)(v),
\\
\,[v,w] & =0,
\endaligned
\end{equation}
for any $x,y \in \h$ and any $v,w \in \j$, {\it is
a Lie bracket in\/} $\g$. One says that the Lie algebra $\g$ so defined
{\it is an Abelian extension of $\h$ associated to the representation
$R:\h\to\gl(\j)$ and the $2$-cocycle\/} $\Lambda$ (see \cite{Che}).
We shall denote by $\h(\Lambda,R)$ the Lie algebra defined 
on the vector space $\g=\h\oplus\j$ with Lie bracket as in \eqref{corchete de ext abeliana}
in terms of the given $2$-cocycle $\Lambda$ with coefficients in the representation $R:\h\to\gl(\j)$. 

\smallskip
\subsection{On the Isomorphism Class of an Abelian extension}

It is well known that for a fixed representation $R:\h\to\gl(\j)$,
and hence, within a fixed cochain complex $C(\h;\j)$, the isomorhism
class of the Abelian extension defined by a $2$-cocylce $\Lambda$
is completely determined by its cohomology class $[\Lambda]$ (see \cite{Che}).
We are interested, however, in understanding how, the isomorphism class
of the Abelian extension defined by $[\Lambda]$, might 
be preserved, even if $[\Lambda]$ is moved into 
a different cohomology class $[\Lambda^\prime] =\gamma.[\Lambda]$  
under a group action that might move the cochain complex
$C(\h;\j)$ into $C(\h;\j)^\prime$ 
by moving the representation $R$ into $R^\prime=\gamma.R$.

\smallskip
Thus, we shall address the question of finding the most general conditions
on a linear map
$\Psi:\h(\Lambda,R)\to \h(\Lambda^\prime,R^\prime)$
to be a Lie algebra isomorphism.
We shall assume, however, that the Abelian ideal $\j$ in $\h(\Lambda,R)$ 
is the same as in $\h(\Lambda^\prime,R^\prime)$;
that is, we shall assume that $\j$ is somehow canonically defined; 
{\it eg\/,} as in \S2.2 below. Under this assumption, $\Psi$ has the form,
\begin{equation}\label{isomorfismo de ext abelianas}
\aligned
\Psi(x) & = g(x) + \Theta(x),\quad x \in \h,\\
\Psi(v) & = \sigma(v),\qquad \, \qquad \,v \in \j,
\endaligned
\end{equation}
where $g:\h\to\h$, $\sigma:\j\to\j$ and $\Theta:\h\to\j$
are linear maps, with $g$ and $\sigma$ invertible.
The isomorphism condition on $\Psi$ is that,
\begin{equation}\label{cond de isomorfismo}
\Psi([x,y])=[\Psi(x),\Psi(y)]^{\prime},\quad \forall \, x,y \in \g,
\end{equation}
where the Lie bracket $[\,\cdot\,,\,\cdot\,]^\prime$ is the one defined in
$\g^{\prime}=\h(\Lambda^\prime,R^\prime)$.
The answer to the question of when is $\h(\Lambda,R)$ isomorphic to
$\h(\Lambda^\prime,R^\prime)$ is given in the following:

\smallskip
\begin{Prop}\label{Prop de isomorfismo de extensiones abelianas}
{\sl Let $\Lambda$ and $\Lambda^\prime$ be $2$-cocyles
within the chain complexes $C(\h;\j)$ and $C(\h;\j)^\prime$, 
associated to the representations $R$ and $R^\prime$ of $\h$ in $\j$, respectively.
Two Abelian extensions $\g=\h(\Lambda,R)$ 
and $\g^\prime=\h(\Lambda^\prime,R^\prime)$ defined in the
underlying vector space $\g=\h\oplus\j$ as in 
\eqref{corchete de ext abeliana} are isomorphic, if and only if,
there are linear maps $g\in\GL(\h)$, $\sigma\in\GL(\j)$ and 
$\Theta\in\Hom_{\F}(\h,\j)$, such that for all $x,y \in \h$ and $v \in \j$:
\begin{equation}\label{condiciones de isomorfismo 1}
\aligned
g([x,y]_{\h})&= [g(x),g(y)]_{\h}
\\
\Theta([x,y]_{\h})+\sigma\left(\Lambda(x,y)\right) & =
\Lambda^\prime(g(x),g(y)) 
\\
&\quad + R^\prime(g(x))(\Theta(y))-R^\prime(g(y))(\Theta(x))
\\
\sigma\left(R(x)(v)\right) & = R^\prime(g(x))(\sigma(v)).
\endaligned
\end{equation}
In other words, if and only if 
$g\in\Aut(\h)$, $\sigma\in\GL(\j)$ and $\Theta\in\Hom_{\F}(\h,\j)$, satisfy,
\begin{equation}\label{condiciones de isomorfismo 2}
(g,\sigma).\Lambda = \Lambda^\prime + \d^\prime(g.\Theta),
\qquad\text{and}\qquad
(g,\sigma).R=R^\prime,
\end{equation}
where, $\d^\prime$ is the differential map of the
cochain complex $C(h;\j)^\prime$ associated to the representation $R^\prime$, and
\begin{equation}\label{act del grupo}
\aligned
g.\Theta \,(\,\cdot\,) & = \Theta\circ g^{-1}\,(\,\cdot\,) ,\\
(g,\sigma).R\,(\,\cdot\,) & = \sigma\circ R\,(\,g^{-1}(\,\cdot\,)\,)\circ \sigma^{-1}, \\
(g,\sigma).\Lambda\,(\,\cdot\,,\,\cdot\,) & =\sigma\left(\Lambda(\,g^{-1}(\,\cdot\,)\,,\,g^{-1}(\,\cdot\,)\,)\right).
\endaligned
\end{equation}
}
\end{Prop}
\begin{proof}
The proof that \eqref{condiciones de isomorfismo 1} follows from 
\eqref{cond de isomorfismo}, is immediate using 
\eqref{isomorfismo de ext abelianas}, whereas 
\eqref{condiciones de isomorfismo 2} simply rewrites 
\eqref{condiciones de isomorfismo 1}
using the obvious group action defined in \eqref{act del grupo} and the well known
definitions of the differential maps on $C(\h;\j)$ and $C(\h;\j)^\prime$ in terms of
$R$ and $R^\prime$, respectively.
\end{proof} 

\smallskip
\begin{Remark}\label{ISO}
{\rm
This result states 
what is needed for a map $\Psi$ to be 
an isomorphism from $\h(\Lambda,R)$ into $\h(\Lambda^\prime,R^\prime)$.
Observe how $\Psi$ depends on the data
$g\in\Aut(\h)$, $\sigma\in\GL(\j)$ and $\Theta\in\Hom_{\F}(\h,\j)$.
We have proved that $\Psi$ is a Lie algebra isomorphism if and only if 
the equations \eqref{condiciones de isomorfismo 1}
(equivalently \eqref{condiciones de isomorfismo 2})
are satisfied, taking into account that the pairs
$(g,\sigma)\in\Aut\h\times\GL(\j)$ act on the data $\{\Theta,\Lambda,R\}$ according to
\eqref{act del grupo}, even though $\Theta$ is a component of $\Psi$ itself.
Let $G=\Aut\h\times\GL(\j)$.
Then $\Theta$ combines with $(g,\sigma)\in G$
so as to produce
the coboundary term shown in \eqref{condiciones de isomorfismo 2}.
In particular, this is consistent with the following factorization of the isomorphism
$\Psi:\h(\Lambda,R)\to \h(\Lambda^\prime,R^\prime)$:
\begin{equation}\label{desc-isomorfismo-1}
\Psi=
\begin{pmatrix}
g & 0 \\
\Theta & \sigma 
\end{pmatrix}
=
\begin{pmatrix}
\operatorname{Id}_{\h} & 0 \\
\Theta\circ g^{-1} & \operatorname{Id}_{\j}
\end{pmatrix}
\circ
\begin{pmatrix}
g & 0 \\
0 & \sigma 
\end{pmatrix},\quad (g,\sigma)\in G.
\end{equation}
Let $\operatorname{Iso}\h(\Lambda,R)$ be the group that
preserves the isomorphism class of the Lie algebra $\h(\Lambda,R)$.
The assignment $\Psi\mapsto (g,\sigma)$ defines a group
epimorphism $\operatorname{Iso}\h(\Lambda,R)\to G$
whose kernel is isomorphic to the Abelian group $\Hom_{\F}(\h,\j)$.
Thus,} 
$$
\operatorname{Iso}\h(\Lambda,R) \simeq G \rtimes \Hom_{\F}(\h,\j).
$$
\end{Remark}

\smallskip
Now, the next result is also a simple and straightforward computation 
from the definitions involved.  

\smallskip
\begin{Prop}\label{Prop de equivariancia}{\sl 
Let $\d$ and $\d^\prime$ be the differential maps
in the cochain complexes $C(\h;\j)$ and $C(\h;\j)^\prime$ defined by the representations 
$R:\h\to\gl(\j)$, and  $R^\prime:\h\to\gl(\j)$, respectively.
Let $\gamma=(g,\sigma)\in G=\Aut(\h)\times\GL(\j)$, and let
$\Phi(\gamma):C(\h;\j)\to C(\h;\j)$ be the $G$-action
given by $\lambda\mapsto \Phi(\gamma)(\lambda) = \gamma.\lambda$,
where 
$$
(\gamma.\lambda)(\,\cdot\,,\ldots,\,\cdot\,)=
\sigma\left(\,\lambda(g^{-1}(\,\cdot\,),\ldots,g^{-1}(\,\cdot\,))\,\right).
$$
Then,
\begin{equation}\label{equivariancia}
\d^\prime\circ\,\Phi(\gamma)=\Phi(\gamma)\circ \d
\quad\Longleftrightarrow\quad
R^\prime = \gamma.R\,,
\end{equation}
where,
$$
\gamma.R\,(x) = \sigma\,\circ\,R\left( g^{-1}(x)\right)\,\circ\,\sigma^{-1},
\quad\forall\,x\in\h.
$$
}
\end{Prop}

\smallskip
\begin{Remark}\label{estabilidad}
{\rm
The property $\d^\prime\circ\,\Phi(\gamma) = \Phi(\gamma)\circ \d$
states that $\d\Lambda = 0$ {\it if and only if\/} $\d^\prime(\Phi(\gamma)\Lambda)=0$.
From the statement of 
{\bf Prop. \ref{Prop de isomorfismo de extensiones abelianas}}, 
we conclude that 
$\h(\Lambda,R)\simeq\h(\Lambda^\prime,R^\prime)$ if and only if
$R^\prime = \gamma.R$ and 
$\Phi(\gamma)\Lambda = \Lambda^\prime + \d^\prime(g.\Theta)$,
which certainly makes $\d^\prime(\Phi(\gamma)\Lambda)=0$
when $\d^\prime\Lambda^\prime=0$. 
That is, the subgroups 
defined by the cocycles and the coboundaries, respectively,
are stable under the group action $\Phi(\gamma)$.
It is also clear that if $R^{\prime}$ is not in the $G$-orbit of $R$, 
then $\h(\Lambda,R)$ cannot be isomorphic to $\h(\Lambda^{\prime},R^{\prime})$.
}
\end{Remark}

\smallskip
The next result deals with a special case of 
{\bf Prop. \ref{Prop de isomorfismo de extensiones abelianas}};
namely, the case when the differential map in the cochain complex $C(\h;\j)$ 
is fixed because the representation $R$ is fixed. 
Most of the classical results for Abelian extensions are obtained within
a single cohomology theory through a fixed representation. By comparing 
{\bf Prop. \ref{Prop de isomorfismo de extensiones abelianas}}
with  {\bf Prop. \ref{Prop equivariancia-isotropia}} below, it is clear
that by restricting the framework to a single cohomology theory, 
there will be several Lie algebras
in the isomorphism class of $\h(\Lambda, R)$ that can never be reached
by changing the cocycle $\Lambda$ in the form,
$\Lambda\mapsto\Lambda^\prime=\Phi(\gamma)\Lambda$
modulo a coboundary. This is the difference between
\eqref{act-cociclo} below and the most general
relationship found in \eqref{condiciones de isomorfismo 2}.

\smallskip
\begin{Prop}\label{Prop equivariancia-isotropia}{\sl
Fix the representation $R:\h\to\gl(\j)$ 
and restrict the group action to pairs $(g,\sigma)$ in the isotropy subgroup $G_R \subset G$
of $R$. Let $\d$ be the differential map of the cochain complex $C(\h;\j)$.
Then
\begin{equation}\label{equivariancia-isotropia}
\d\circ\,\Phi(\gamma) = \Phi(\gamma)\circ \d
\quad\Longleftrightarrow\quad
\gamma=(g,\sigma)\in G_R\,;
\end{equation}
that is, $\d$ is $G_R$-equivariant. In particular, if
$\Lambda$ and $\Lambda^\prime$ are $2$-cocylces, then
$\h(\Lambda,R)\simeq\h(\Lambda^\prime,R)$
if and only if there are maps $g\in\Aut(\h)$, $\sigma\in\GL(\j)$ and $\Theta\in\Hom(\h,\j)$,
such that,
\begin{equation}\label{act-cociclo}
\Phi(\gamma)(\Lambda)=\Lambda^\prime + \d(g.\Theta),
\qquad\text{with}\quad
\gamma=(g,\sigma)\in G_R.
\end{equation}
}
\end{Prop}

\smallskip
\section{Extensions Defined by Two Canonical Abelian Ideals}

\subsection{Isomorphisms of Abelian Extensions Defined by Two Can\-onical Ideals}

Let $\g$ be a Lie algebra with Lie bracket $[\,\cdot\,,\,\cdot\,]$.
We shall show in \S{2.2} below how to define
(under the special hypothesis that $\g$ has a descending central ideal)
two characteristic ideals, $\ide=\ide(\g)$ and $\j=\j(\g)$ of $\g$,
satisfying the following properties:
\begin{equation}\label{condiciones para los ideales}
\text{(a)} \ \ \ \j\ \ \text{is abelian;} \qquad\quad
\text{(b)} \ \ \  \ide \subset \j; \qquad\quad
\text{(c)} \ \ \  [\g,\j] \subset \ide.
\end{equation}
We shall also see in  {\bf Lemma \ref{lema auxiliar 2}} that if $\g$ admits an invariant metric, 
then $\ide^{\perp}=\j$. For the time being, however, we shall first restrict ourselves
to the properties (a), (b), and (c) in \eqref{condiciones para los ideales}.

\smallskip
Let $\h=\g/\j$, and decompose $\g$
in the form $\g=\h\oplus \j$. We shall also assume that $\ide\ne\j$,
and therefore, $\a = \j/\ide\ne\{0\}$. Thus, we may further decompose $\g$
in the form $\g=\h\oplus\a\oplus\ide$, and write  its elements as,
$$
\g\ni x+v+\theta;
\qquad x\in\h\,,\ v\in \a\,,\ \theta\in \ide\,,\ v+\theta\in \j.
$$
Using the fact that $\j$, and hence $\ide$ by (b) in
\eqref{condiciones para los ideales}, are Abelian, we have,
\begin{equation}\label{corchete en terminos de los ideales}
[x+u+\theta, y+v+\eta] = [x,y]+[x,v+\eta]-[y,u+\theta],
\end{equation}
where $[x,y]\in\g=\h\oplus \j$, and $[x,v+\eta]$ and $[y,u+\theta]$ belong to $\ide$ 
because of (c) in \eqref{condiciones para los ideales}. 
Now, decompose $\Lambda(x,y)$ in the form,
$\Lambda(x,y)=\lambda(x,y)+\mu(x,y)$, with 
$\lambda(x,y)\in\a$ and $\mu(x,y)\in\ide$, respectively. 
Moreover, the representation $R:\h\to\gl(\j)=\gl(\a\oplus\ide)$
also decomposes by means of (c), into
$[x,v]=\varphi(x)(v)\in\ide$ and $[x,\theta]=\rho(x)(\theta)\in\ide$, 
respectively, for any $x\in\h$.
That is, \eqref{corchete en terminos de los ideales} has now the following 
finer structure produced by the ideals $\ide$ and $\j$:
\begin{equation}\label{corchete en terminos de los ideales 2}
{\aligned
\,[x,y] & = [x,y]_{\h} + \lambda(x,y) + \mu(x,y),
\\
[x,v] & = \varphi(x)(v),
\\
[x,\theta] & = \rho(x)(\theta),
\\
[v,w] & = [v,\eta] = [\theta,\eta] = 0,
\endaligned}
\end{equation}
for any $x,y  \in \h$, $v,w \in \a$ and $\theta,\eta \in \ide$.
In particular, it follows that, for each $x\in\h$, 
$R(x):\a \oplus \ide \to \a \oplus \ide$, is given by
\begin{equation}\label{desc de la rep R}
R(x) = \begin{pmatrix} 0 & 0 \\
\varphi(x) & \rho(x)\end{pmatrix}
:
\begin{pmatrix} v \\
\theta \end{pmatrix}
\mapsto
\begin{pmatrix} 0 \\
\varphi(x)(v) + \rho(x)(\theta)\end{pmatrix}.
\end{equation}
Moreover, the property \eqref{representacion R} yields the following identities:
\begin{equation}\label{condiciones para la desc de R}
\aligned
\varphi([x,y]) & = 
\varphi([x,y]_{\h}) =
\rho(x)\circ\varphi(y)-\rho(y)\circ\varphi(x),
\\
\rho([x,y]) & = 
\rho([x,y]_{\h}) =
\rho(x)\circ\rho(y)-\rho(y)\circ\rho(x),
\endaligned
\end{equation}
for all $x,y \in \h$. The equalities from the middle terms to the left hand sides
follow from $R([x,y])=R([x,y]_{\h})$
which is a consequence of the fact that the ideal $\j$ is Abelian.
We shall use, however, the equalities from the middle terms
to the right hand sides. They state that
$\rho:\h\to\gl(\ide)$ {\it is a representation of\/} $\h$ and that
$\varphi$ {\it is a $1$-cocycle in $C\left(\h;\Hom_{\F}(\a,\ide)\right)$
for the representation $\bar{\rho}:\h\to\gl(\Hom_{\F}(\a,\ide))$
defined by\/} $\bar{\rho}(x)(\tau)=\rho(x)\circ \tau$ on cochains
$\tau:\h\times\cdots\times\h\to\Hom_{F}(\a,\ide)$. 
Observe that the representation $\bar{\rho}:\h \to \gl(\Hom_{\F}(\a,\ide))$,
is no other than the natural 
tensor product representation in 
$\ide\otimes \a^{*}\simeq \Hom_{\F}(\a,\ide)$, 
when $\a$ is the trivial $\h$-module.

\smallskip
These properties on $\varphi$ and $\rho$ are needed for making $\g$ into a Lie algebra.
Indeed, in order to look at the information contained in Jacobi identity, one may
compute Lie brackets of the form,
$$
\left[\,[\,x+u+\theta\,,\,y+v+\eta\,]\,,\,z+w+\xi\,\right],
$$
and then take the corresponding cyclic sum.
It is a straightforward matter to show that the Lie bracket above is equal to,
$$
\aligned
\left[\,[\,x\,,\,y\,]\,,\,z\,\right] & +  \varphi([x,y])(w) + \rho([x,y])(\xi)
+
\rho(z)(\varphi(y)(u))
\\
& \quad 
+\rho(z)(\rho(y)(\theta))-\rho(z)(\varphi(x)(v))-\rho(z)(\rho(x)(\eta))
\endaligned
$$
The corresponding cyclic sum of three terms like this
involves Jacobi identity for the Lie bracket in 
$\g$ of three elements in $\h$ and cyclic sums over the triples $\{(x,u,\theta),(y,v,\eta),(z,w,\xi)\}$
of terms similar to the last six in this expression.
It is easy to verify that one is left with a sum of terms belonging to $\ide$ which
vanish identically because of the properties of the representation $R$
of $\g$ just observed in terms of $\varphi$ and $\rho$.

\smallskip
On the other hand, by writing down $\left[\,[\,x\,,\,y\,]\,,\,z\,\right]$ 
for the triple $\{x,y,z\}$ of elements from $\h$, but this time in terms
of the $\h$-component $[\,\cdot\,,\,\cdot\,]_{\h}$
of the Lie bracket $[\,\cdot\,,\,\cdot\,]$ in $\g$, the Jacobi identity
will produce cyclic sums of three different expressions 
corresponding to the components along the
direct sumands $\h$, $\a$ and $\ide$, since,
$$
\aligned
\left[\,[\,x\,,\,y\,]\,,\,z\,\right]
& = 
\left[\,
[x,y]_{\h}+\lambda(x,y)+\mu(x,y)\,,\,z
\,\right]
\\ 
& = 
\left[\,
[x,y]_{\h}\,,\,z
\,\right]_{\h}
+ \lambda([x,y]_{\h} \, ,z) + \mu([x,y]_{\h} \, ,z)
\\
&\quad
-
\left[\,
z\,,\,\lambda(x,y)+\mu(x,y)
\,\right]
\\
& = 
\left[\,
[x,y]_{\h}\,,\,z 
\,\right]_{\h}
+ \lambda([x,y]_{\h}\, ,z) 
\\
&\quad 
+ \mu([x,y]_{\h}\, ,z) - \varphi(z)(\lambda(x,y)) - \rho(z)(\mu(x,y)).
\endaligned
$$
We therefore end up with,
\begin{equation}\label{desc de corchete en terminos de los ideales}
\aligned
\sum_{\circlearrowleft\{x,y,z\}} 
&
\!\!\!\left[\,
[x,y]_{\h}\,,\,z
\,\right]_{\h}
=0,\qquad\qquad
\sum_{\circlearrowleft\{x,y,z\}}
\!\!\!\lambda([x,y]_{\h}\, ,z)
=0,
\\
\sum_{\circlearrowleft\{x,y,z\}}
&
\!\!\!\left\{\,
\mu([x,y]_{\h}\, ,z)
- \varphi(z)\!\left(\lambda(x,y)\right) - \rho(z)(\mu(x,y))
\,\right\}
=0.
\endaligned
\end{equation}
The first is just Jacobi identity for the Lie algebra $\h=\g/\j$.
The second one is easy to understand for the
skew-symmetric bilinear map $\lambda:\h\times\h\to\a$
taking values in the {\it trivial $\h$-module\/} $\a$. In fact,
ordinary Lie algebra cohomology lets us write
$$
(\d\lambda) (x,y,z) = -\!\!\!\sum_{\circlearrowleft\{x,y,z\}}
\!\!\!\lambda([x,y]_{\h},z) = 0.
$$
On the other hand, for the
skew-symmetric bilinear map $\mu:\h\times\h\to\ide$
into the $\h$-module defined by the
representation $\rho:\h\to\gl(\ide)$, we have,
$$
(\d\mu) (x,y,z) =\!\!\!\sum_{\circlearrowleft\{x,y,z\}}
\!\!\!\left(
\,\rho(z)(\mu(x,y)) -\mu([x,y]_{\h},z)\,
\right).
$$
In particular, the third cyclic sum in 
\eqref{desc de corchete en terminos de los ideales} 
states that,
\begin{equation}\label{interpetacion1}
(\d\mu) (x,y,z) + \!\!\!\sum_{\circlearrowleft\{x,y,z\}}
\!\!\!\varphi(z)\left(\lambda(x,y)\right)=0.
\end{equation}
We may interprete the term
$\sum_{\circlearrowleft}\varphi(z)(\lambda(x,y))$
as the result of applying a map 
$e_{\varphi}:C^2(\h;\a)\to C^3(\h;\ide)$, induced by
$\varphi:\h\to\Hom_{\F}(\a,\ide)$ 
on the  chain complexes involved, as follows:
$$
C^2(\h;\a)\ni\lambda(\,\cdot\,,\,\cdot\,)\ \mapsto\ 
e_{\varphi}(\lambda)(x,y,z)=
\!\!\!\sum_{\circlearrowleft\{x,y,z\}}
\!\!\!\varphi(x)(\lambda(y,z))\in C^3(\h;\ide),
$$
where the cyclic sum in the right hand side is taken over
the arguments, thus producing an alternating trilinear
map $e_{\varphi}(\lambda):\h\times\h\times\h\to\ide$ from the initial alternating map
$\lambda:\h\times\h\to\a$. 
In general, one may define a {\it degree-one map $e_{\varphi}$
of cochain complexes\/} by means of,
$$
C^n(\h;\a) 
\ni \lambda \ \mapsto \ e_{\varphi}(\lambda) \in C^{n+1}(\h;\ide),
$$
where, for any $x_1,\ldots, x_{n+1}$ in $\h$,
\begin{equation}
\label{operador e}
e_{\varphi}(\lambda)(x_1,\ldots,x_{n+1})  =
\sum_{i=1}^{n+1}
(-1)^{i+1}\varphi(x_i)(\lambda(x_1,\ldots,\widehat{x_{i}},\ldots, x_{n+1})).
\end{equation}
The following result elucidates the behavior of $e_{\varphi}$ with respect to the 
corresponding differential maps on $C(\h;\a)$ and $C(\h;\ide)$ which,
for the statement and proof, we shall denote by $\d_{\a}$ and $\d_{\ide}$, respectively.

\smallskip
\begin{Prop}\label{anticonmutativo}{\sl
Let $(\h,[\cdot,\cdot]_{\h})$ be a Lie algebra. Let $\a$
be a finite-dimensional trivial $\h$-module and let $\ide$
be the $\h$-module given by the representation $\rho:\h\to\gl(\ide)$.
Let $\bar{\rho}:\h \to \gl(\Hom_{\F}(\a,\ide))$ be the
tensor product representation, so that $\bar{\rho(x)}(T)=\rho(x)\,\circ\,T$,
for any $x \in \h$ and any $T \in \Hom_{\F}(\a,\ide)$.
Let $\varphi \in C(\h;\Hom_{\F}(\a,\ide))$ be a 1-cocycle with coefficients
in the representation $\bar{\rho}$. The degree-one map of cochain complexes
$e_{\varphi}:C(\h;\a) \to C(\h;\ide)$ defined by \eqref{operador e},
satisfies, 
$$
(\,e_{\varphi} \,\circ \d_{\a}\,)\,\vert_{\,C^{n-1}(\h;\a)}
=-\,(\,\d_{\ide}\, \circ\,\, e_{\varphi}\,)\,\vert_{\,C^{n-1}(\h;\a)},
$$
for each $n\in\N$; that is, the following diagram anticommutes:
$$
\xymatrix{ 
C^{n-1}(\h;\a){\ar[r]^{\d_{\a}}}{\ar[d]_{e_{\varphi}}} & C^n(\h;\a){\ar[d]^{e_{\varphi}}}\\
C^n(\h;\ide){\ar[r]^{\d_{\ide}}} & C^{n+1}(\h;\ide).
}
$$
} 
\end{Prop}
\begin{proof}
Consider $\j=\a \oplus \ide$, and the representation $R:\h \to \gl(\j)$, given by,
$R(x)(v+\theta)=\varphi(x)(v)+\rho(x)(\theta)$, $\forall\, v \in \a$ and $\forall\, \theta \in \ide$
(see \eqref{desc de la rep R}).
The differential map $\d$ 
in the cochain complex $C(\h;\j)$ defined by
the representation $R:\h \to \gl(\j)$, satisfies
$$
\d=\iota_{\a} \circ\,\d_{\a}+\iota_{\ide}\circ\, (e_{\varphi} \oplus \d_{\ide})
\leftrightarrow (\d_{\a},e_{\varphi}\oplus\d_{\ide}),
$$
where $\iota_{\a}:\a \to \j$ and $\iota_{\ide}:\ide \to \j$ are the inclusion maps.
The decomposition $\j=\a \oplus \ide$ makes
$C(\h;\j)$ to decompose as $C(\h;\a) \oplus C(\h;\ide)$
and therefore, the cochain complex $(C(\h;\j),\d)$ becomes isomorphic
to $\left(\,C(\h;\a) \oplus C(\h;\ide),(\d_{\a},e_{\varphi}\oplus\d_{\ide})\right)$.
Then,
$$
(\,e_{\varphi} \,\circ \d_{\a}\,)\,\vert_{\,C^{n-1}(\h;\a)}
=-\,(\,\d_{\ide}\, \circ\,\, e_{\varphi}\,)\,\vert_{\,C^{n-1}(\h;\a)},
\quad\forall\,n\in\N,
$$
follows from the fact that $(\,\d\,\circ\,\d)\,\vert_{C^{n}(\h;\j)}=0$.
\end{proof}

\smallskip
\begin{Remark}{\rm
Write $\Lambda=\lambda\oplus\mu$, for any $\Lambda \in C(\h;\j)$, 
with $\lambda \in C(\h;\a)$ and $\mu \in C(\h;\ide)$. 
In view of {\bf Prop. \ref{anticonmutativo}} 
and the specific form of the representation  $R$ 
given in \eqref{desc de la rep R} in terms of $\varphi$ and $\rho$,
we shall write,
$$
\aligned
(\d\Lambda)(x_1,\ldots,x_{n+1})=&\sum_{i=1}^{n+1}(-1)^{i+1}R(x_{i})
(\Lambda(x_1,\ldots,\widehat{x_{i}},\ldots,x_{n+1}))\\
\,& +\sum_{i<j}(-1)^{i+j}
\Lambda([x_{i},x_{j}]_{\h},x_1,\ldots,\widehat{x_{i}},\ldots,\widehat{x_{j}},\ldots,x_{n+1})\\
=&
(\d_{\a}\lambda)(x_1,\ldots,x_{n+1})\\
\,&\oplus\,\left(\,e_{\varphi}(\lambda)(x_1,\ldots,x_{n+1})
+(\d_{\ide}\mu)(x_1,\ldots,x_{n+1})\,\right)\\
\leftrightarrow & \ 
\begin{pmatrix} (\d_{\a}\lambda)(x_1,\ldots,x_{n+1}) \\
e_{\varphi}(\lambda)(x_1,\ldots,x_{n+1})+(\d_{\ide}\mu)(x_1,\ldots,x_{n+1})
\end{pmatrix}
\\
=&
\begin{pmatrix}
\d_{\a} & 0 \\ e_{\varphi} & \d_{\ide}
\end{pmatrix}
\begin{pmatrix}
\lambda \\
\mu
\end{pmatrix}
(x_1,\ldots,x_{n+1})\\
=& 
\left(
\D_{\varphi}
\begin{pmatrix}
\lambda \\
\mu
\end{pmatrix}
\right)(x_1,\ldots,x_{n+1})\,.
\endaligned
$$
In other words, the differential map $\d$ of $C(\h;\j)$ gets identified with 
the operator $\D_{\varphi}=\left(\begin{smallmatrix}\d_{\a}&0\\ 
e_{\varphi}&\d_{\ide}\end{smallmatrix}\right)$
that acts on $C(\h;\a)\oplus C(\h;\ide)$. 
From now on we shall
omit the explicit reference to $\a$ in $\d_{\a}$ 
and to $\ide$ in $\d_{\ide}$ and simply write $\d$,
as their meaning is clear from the context of the operator $\D_{\varphi}$.}
\end{Remark}

\smallskip
We may now summarize what we have done so far in this section
in the following statement:

\smallskip
\begin{Cor}\label{corolario ext abelianas}{\sl
Define a skew-symmetric bilinear map
$[\,\cdot\,,\,\cdot\,]:\g\times\g\to\g$ 
on the underlying vector space $\g=\h\oplus\a\oplus\ide$, 
by means of \eqref{corchete en terminos de los ideales 2},
where $\lambda:\h\times\h\to \a$ and $\mu:\h\times\h\to \ide$ 
are $2$-cochains in
the complexes $C(\h;\a)$ and $C(\h;\ide)$ associated to the
trivial representation of $\h$ in $\a$ 
and to the representation $\rho$ of $\h$ in $\ide$, respectively.
Let $\varphi:\h\to\Hom_{\F}(\a,\ide)$ be a $1$-cochain in the complex $C(\h;\Hom_{\F}(\a,\ide))$
associated to the representation $\bar{\rho}:\h\to\gl(\Hom_{\F}(\a,\ide))$
defined by $\bar{\rho}(x)(\tau)=\rho(x)\circ\tau$.
Then, $[\,\cdot\,,\,\cdot\,]$ is a Lie algebra bracket on $\g$
if and only if
\begin{equation}\label{diferenciales}
\qquad\d\lambda = 0,\qquad\  \d\mu+e_{\varphi}(\lambda) = 0,\qquad\  \d\varphi = 0,
\end{equation}
where $e_\varphi:C(\h;\a)\to C(\h;\ide)$ is the degree-one map
of cochain complexes defined in \eqref{operador e}. Moreover,
\eqref{diferenciales} can be rewritten in terms of the differential map $\D_{\varphi}
=\left(\begin{smallmatrix}\d&0\\ e_{\varphi}&\d\end{smallmatrix}\right)$
acting on $C(\h;\a)\oplus C(\h;\ide)\to C(\h;\a)\oplus C(\h;\ide)$, so that,
$[\,\cdot\,,\,\cdot\,]$ is a Lie algebra bracket on $\g$ if and only if
\begin{equation}\label{diferencial D}
\D_{\varphi}
\begin{pmatrix}\lambda \\ \mu\end{pmatrix}=\begin{pmatrix} 0 \\ 0\end{pmatrix}
\qquad\text{and}\qquad\  \d\varphi = 0.
\end{equation}
}
\end{Cor}

\smallskip
{\bf Note.} 
Consider the differential maps
$\d:C(\h;\j)\to C(\h;\j)$  and  $\d^\prime:C(\h;\j)\to C(\h;\j)$
associated to the representations $R$ and $R^\prime=(g,\sigma).R$,
respectively, with $g\in\Aut(\h)$ and $\sigma\in\GL(\j)$.
We shall use the finer decomposition $\j=\a\oplus\ide$
and will assume that the ideals $\ide$ and $\j$ are Abelian,
are canonically defined, and are such that $\ide\subset\j$ and 
$[\g,\j]\subset \ide$, so that, any isomorphism $\h\oplus\a\oplus\ide
\to \h\oplus\a\oplus\ide$ must map $\j$ into $\j$ and $\ide$ into itself. Thus, 
in {\bf Prop. \ref{Prop de isomorfismo de extensiones abelianas}}, we must have,
\begin{equation}\label{sigma para los ideales canonicos}
\sigma = 
\begin{pmatrix}
h & 0 \\ 
T & k
\end{pmatrix},
\quad\text{with,}\quad
h\in\GL(\a),\  k\in\GL(\ide),\ T\in\Hom_{\F}(\a,\ide).
\end{equation}
and we shall also write, 
$\Theta(x) = \tau(x)\oplus\nu(x)\leftrightarrow \left(\begin{smallmatrix}\tau(x) \\ 
\nu(x)\end{smallmatrix}\right)$
for the map $\Theta:\h \to \a \oplus \ide$, in
{\bf Prop. \ref{Prop de isomorfismo de extensiones abelianas}}.
Taking into account the finer structure brought by the decomposition
$\j=\a\oplus\ide$, the specific form of $\sigma$ in \eqref{sigma para los ideales canonicos}
and the dependence of $\Lambda$ and $R$ on the data $(\lambda,\mu)$ and
$(\varphi,\rho)$, respectively, 
we shall write $\h(\lambda,\mu,\varphi,\rho)$ instead of $\h(\Lambda,R)$
and state the following {\bf Corollary} to 
{\bf Prop. \ref{Prop de isomorfismo de extensiones abelianas}}
and {\bf Prop. \ref{Prop de equivariancia}}:

\begin{Cor}\label{equivariancia de ecuaciones}{\sl
Suppose $\g=\h(\lambda,\mu,\varphi,\rho)$ and
$\g^\prime=\h(\lambda^\prime,\mu^\prime,\varphi^\prime,\rho^\prime)$
are two Lie algebras whose Lie brackets $[\,\cdot\,,\,\cdot\,]$ and $[\,\cdot\,,\,\cdot\,]^\prime$
are defined on the underlying vector space $\h\oplus\a\oplus\ide$, 
in terms of the data satisfying the conditions of 
{\bf Prop. \ref{Prop de isomorfismo de extensiones abelianas}} 
and {\bf Prop. \ref{Prop de equivariancia}}. 
Furthermore, assume that
$\ide$ and $\j=\a\oplus\ide$ are two canonically defined 
Abelian ideals satisfying
$[\g,\j]\subset\ide$ and $[\g^\prime,\j]^\prime\subset\ide$.
Then, $\g$ and $\g^\prime$ are isomorphic if and only if there exist
$g\in\Aut(\h)$, $h\in\GL(\a)$, $k\in\GL(\ide)$ and 
linear maps $\tau:\h\to\a$, $\nu:\h\to\ide$ and $T:\a\to\ide$ such that
\begin{equation}\label{9}
\begin{split}
& \Phi(\gamma)\begin{pmatrix} \lambda \\ \mu \end{pmatrix}
=
\begin{pmatrix} \lambda^\prime \\ \mu^\prime \end{pmatrix}
+ \begin{pmatrix} \d & 0 \\ e_{\varphi^\prime} & \d^\prime \end{pmatrix}
\begin{pmatrix} \tau\circ g^{-1} \\ \nu\circ g^{-1} \end{pmatrix}
=:
\begin{pmatrix} \lambda^\prime \\ \mu^\prime \end{pmatrix}
+
\D^\prime_{\varphi^\prime}
\begin{pmatrix} \tau\circ g^{-1} \\ \nu\circ g^{-1} \end{pmatrix},\\
&\Phi(\gamma)\,\varphi = \varphi^\prime + \d(T\circ h^{-1}) \qquad\text{\it and}\qquad
\Phi(\gamma)\,\rho  = \rho^\prime,
\end{split}
\end{equation}
where $\gamma=(g,\sigma)$ and $\sigma
=\left(\begin{smallmatrix} h& 0 \\ T & k \end{smallmatrix}\right)$, 
just as in \eqref{sigma para los ideales canonicos}, and
$$
\aligned
\left(\Phi(\gamma)\begin{pmatrix} \lambda \\ \mu \end{pmatrix}\right)(x,y) & =
{\begin{pmatrix}
h & 0 \\ 
T & k
\end{pmatrix}
\begin{pmatrix}
\lambda(g^{-1}(x),g^{-1}(y))
\\
\mu(g^{-1}(x),g^{-1}(y))
\end{pmatrix},}
\\
\left(\,\Phi(\gamma)\,\varphi\,\right)(x) & =
k\circ\varphi\left(g^{-1}(x)\right)\circ h^{-1},
\\
\left(\,\Phi(\gamma)\,\rho\,\right)(x) & =
k\circ\rho\left(g^{-1}(x)\right)\circ k^{-1}.
\endaligned
$$
Moreover
$
\D^\prime_{\varphi^\prime}\circ\,\Phi(\gamma) = \Phi(\gamma)\,\circ\,\D_{\varphi}$,
if and only if the relations \eqref{9} hold true.}
\end{Cor}

\smallskip
Fix the representation $R=(\varphi,\rho):\h \to \gl_{\F}(\j)$.
Then, {\bf Prop. \ref{anticonmutativo}} and {\bf Cor. \ref{corolario ext abelianas}}
give the short exact sequence of complexes,
\begin{equation}\label{sucesion de complejos}
\xymatrix{
0  \ar[r] & C(\h;\ide)  \ar[r] & C(\h;\a) \oplus C(\h;\ide)  \ar[r] & C(\h;\a)  \ar[r]  & 0
}
\end{equation}
where $e_{\varphi}:C(\h;\a) \to C(\h;\ide)$ is the connecting homomorphism. 
Denote the respective coboundaries, cocycles and cohomology group
of the complex $C(\h;\a) \oplus C(\h;\ide)$ in the classical way:
$$
B_R(\h;\a \oplus \ide),\quad Z_R(\h;\a \oplus \ide), \quad H_R(\h;\a \oplus \ide).
$$
It follows from \textbf{Cor. \ref{corolario ext abelianas}}
that a pair $(\lambda,\mu) \in C(\h;\a) \times C(\h;\ide)$ 
gives rise to a Lie algebra $\h(\lambda,\mu,\varphi,\rho)$
if and only if $(\lambda,\mu) \in Z_R(\h;\a \oplus \ide)$.
Using now {\bf Cor. \ref{equivariancia de ecuaciones}} 
we may deduce the following expected
criterion to determine whether two Lie algebras having
the same underlying representation $R=(\varphi,\rho)$ are isomorphic or not.

\smallskip
\begin{Cor}\label{criterio cohomologico}{\sl
Let $\g=\h(\lambda,\mu,\varphi,\rho)$ and 
$\g^{\prime}=\h(\lambda^{\prime},\mu^{\prime},\varphi,\rho)$ 
be two Lie algebras constructed
as in {\bf Cor. \ref{corolario ext abelianas}}
in terms of the same representation $R=(\varphi,\rho)$.
Then, $\g$ and $\g^{\prime}$ are isomorphic if and only if
$(\lambda,\mu)$ and $(\lambda^\prime,\mu^\prime)$
define the same cohomology class in $H_R^2(\h;\a \oplus \ide)$.
}
\end{Cor}

\smallskip
\begin{proof}
Choose 
$g=\operatorname{Id}_{\h}$, $h=\operatorname{Id}_{\a}$, 
$T=0$ and $k=\operatorname{Id}_{\ide}$
for the isomorphism data given in \textbf{Cor. \ref{equivariancia de ecuaciones}}
for $\Psi$. Then, by \eqref{9} we have,
\begin{equation}\label{criterio cohomologico 1}
\begin{split}
\lambda^{\prime}&=\lambda-\d\tau,\\
\mu^{\prime}&=\mu-e_{\varphi^{\prime}}(\tau)-\d\nu,\\
\varphi^{\prime}&=\varphi,\\
\rho^{\prime}&=\rho.
\end{split}
\end{equation}
In particular, $R=(\varphi,\rho)=(\varphi^{\prime},\rho^{\prime})=R^{\prime}$.
In addition, using the differential operator $\d_{\varphi}$
(see \eqref{diferenciales} and \eqref{diferencial D}), we obtain
$(\lambda^{\prime},\mu^{\prime})=(\lambda,\mu)+\d_{\varphi}(-\tau,-\nu)$.
Thus, $(\lambda,\mu)\equiv (\lambda^{\prime},\mu^{\prime})\mod B_R^2(\h;\a \oplus \ide)$
as in \eqref{criterio cohomologico 1} above and the chosen 
$\Psi:\h(\lambda,\mu,\varphi,\rho) \to \h(\lambda^{\prime},\mu^{\prime},\varphi,\rho)$ 
gives the desired isomorphism.
The converse is clear.
\end{proof}

\smallskip
We may now close this section by showing how to produce
in a canonical way 
the abelian ideals $\ide$ and $\j$
of $\g$ that satisfy \eqref{condiciones para los ideales}.

\subsection{
Definition of the Canonical Ideals $\ide(\g)$ and $\j(\g)$ for a Lie Algebra $\g$
having at least one Abelian descending central ideal.
}

\medskip
\noindent
We shall adhere ourselves to the standard convention of writing,
$$
C(\g)=\{z\in\g\mid [z,x]=0,\ \text{for all}\ x\in\g\,\},
$$
for {\it the center\/} of a Lie algebra 
$\g$ with Lie bracket $[\,\cdot\,,\,\cdot\,]$. 
The {\it descending central series} 
$\g^0\supset\g^1\supset\cdots\supset \g^{\ell-1}\supset\g^\ell\supset\cdots$
of $\g$ is defined by,
$$
\g^0=\g, \qquad \g^\ell=[\g,\g^{\ell-1}], \qquad  \ell \in \N.
$$
The {\it derived series\/} of $\g$ is defined by,
$$
\g^{(0)}=\g,
 \qquad 
\g^{(\ell)}=[\g^{(\ell-1)},\g^{(\ell-1)}],\qquad  \ell \in \N,
$$
and the {\it derived central series\/} 
$C_1(\g)\subset C_2(\g)\subset\cdots \subset C_{\ell}(\g)\subset\cdots$ 
of $\g$ is defined by,
$$
C_1(\g)=C(\g), \qquad C_{\ell}(\g)=\pi_{\ell-1}^{-1}C(\g/C_{\ell-1}(\g)),
\qquad  \ell \in \N,\ \ell>1,
$$
where $\pi_{\ell-1}:\g \to \g/C_{\ell-1}(\g)$ is the corresponding canonical projection.
A Lie algebra $\g$ is \textbf{solvable} if there exists
an $\ell \in \N$ such that $\g^{(\ell)}=\{0\}$. 
Abelian Lie algebras are solvable with $\ell=1$.

\smallskip
The first result in this section {\it is\/} well known. 
Actually, we may refer the reader to \cite{BenayadiSuper}, \cite{Med} or \cite{Zhu}
for at least the fact that  for any $\ell\in\N$, $C(\g^{\ell-1})=C_{\ell}(\g)=(\g^{\ell})^{\perp}$ 
when  $\g$ admits an invariant metric. 
The main statement $C_{\ell}(\g) \subseteq C(\g^{\ell-1})$ holds true in general 
with no need of any invariant metric at all and it can be proved in a straightforward 
manner by induction on $\ell$. We may safely omit the details.

\begin{Prop}\label{proposicion auxiliar}{\sl
Let $\g$ be a Lie algebra with Lie bracket $[\cdot,\cdot]$ 
and let $C_{\ell}(\g)$ be the $\ell$-th ideal in its derived central series. Then,  
$$
C_{\ell}(\g) \subseteq C(\g^{\ell-1}):=\{x \in \g\,|\,[x,\g^{\ell-1}]=\{0\}\},\,\,\forall \ell \in \N.
$$
Furthermore, if $\g$ admits an invariant metric,
then $C_{\ell}(\g)=C(\g^{\ell-1})=(\g^{\ell})^{\perp}$, for all $\ell \in \N$.}
\end{Prop}

\smallskip
The following {\bf Lemma} defines in a canonical way
the two ideals, $\ide(\g)$ and $\j(\g)$, whose properties form the basis of this work.

\smallskip
\begin{Lemma}\label{lema auxiliar 2}{\sl
Let $\g$ be an Lie algebra over a field of characteristic zero
and let $[\,\cdot\,,\,\cdot\,]$ be its Lie bracket. 
Define,
$$
\ide(\g)=\sum_{k \in \N}C_k(\g) \cap \g^k \ \quad\text{and}\quad\ \j(\g)=\bigcap_{k \in \N}(C_k(\g)+\g^k).
$$
Then,
\begin{itemize}

\item[(i)] $\ide(\g) \subset \j(\g)$.

\item[(ii)] 
There exists  an $m\ge 1$, such that,
$$
\j(\g)=C(\g)+\sum_{k=1}^{m-1}C_{k+1}(\g) \cap \g^k+\g^m.
$$

\item[(iii)] 
$\j(\g)$ is Abelian if and only if there is an $\ell\ge 1$, such that $\g^{\ell}$ is Abelian.

\item[(iv)]
If $\g$ admits an invariant metric and 
$\j(\g)$ is Abelian, then $[\g,\j(\g)] \subset \ide(\g)$ and $\ide(\g)^{\perp}=\j(\g)$.

\item[(v)] 
If $\g$ is a solvable
non-Abelian quadratic Lie algebra having one Abelian 
descending central ideal, then $\ide(\g) \neq \{0\} \neq \j(\g)$.

\end{itemize}
}
\end{Lemma}

\smallskip
\begin{proof}
{\bf (i)} Let $j,k,\ell \in \N$ be such that $j \leq k \leq \ell$.
Then $C_k(\g) \cap \g^k\subset \g^k \subset \g^j \subset \g^j +C_{j}(\g)$. 
Similarly, $C_k(\g) \cap \g^k \subset C_{k}(\g) \subset C_{\ell}(\g) 
\subset C_{\ell}(\g)+\g^{\ell}$. 
Thus, $C_k(\g) \cap \g^k \subset C_{\ell}(\g)+\g^{\ell}$, for all $k, \ell \in \N$. 
Whence, $C_k(\g) \cap \g^k \subset \displaystyle{\bigcap_{\ell \in \N}(C_{\ell}(\g)+\g^{\ell})}$,
and therefore, $\displaystyle{\sum_{k \in \N}(C_k(\g) \cap \g^k) 
\subset \bigcap_{\ell \in \N}(C_{\ell}(\g)+\g^{\ell})}$.
\smallskip

{\bf (ii)} By definition, 
$\j(\g)=(C_1(\g)+\g^1)\bigcap\biggl(\displaystyle{\bigcap_{k \geq 2}(C_k(\g)+\g^k)}\biggr)$.
Observe that 
$C_1(\g) \subset C_k(\g)+\g^k$, for all $k \geq 2$. 
Thus, $C_1(\g) \subset \displaystyle{\bigcap_{k \geq 2}(C_k(\g)+\g^k)}$. 
It is a well known fact that for any triple of vector subspaces $U$, $V$ and $W$ of $\g$,
$U\subset V\Rightarrow (U+W)\cap V = U+W\cap V$. We shall refer to this result
as {\it the $+$ $\cap$ distribution property\/.}
In particular, 
$$
\aligned
\j(\g)&=C_1(\g)+\g^1 \bigcap \biggl(\bigcap_{k \geq 2}(C_k(\g)+\g^k)\biggr)\\
\,&=C_1(\g)+\g^1 \bigcap (C_2(\g)+\g^2) \bigcap\biggl(\bigcap_{k \geq 3}(C_k(\g)+\g^k\biggr).
\endaligned
$$
Now, $C_2(\g) \subset C_k(\g)+\g^k$ ($\forall\,k \geq 3$) $\Rightarrow$ 
$C_2(\g) \subset \displaystyle{\bigcap_{k \geq 3}(C_k(\g)+\g^k)}$. 
Thus, we apply the $+$ $\cap$ distribution property to 
$(C_2(\g)+\g^2) \bigcap\biggl(\displaystyle{\bigcap_{k \geq 3}(C_k(\g)+\g^k}\biggr)$
and get,
$$
\j(\g)=C_1(\g)+\g^1 \bigcap 
\biggl( C_2(\g)+\g^2 \bigcap\biggl(\displaystyle{\bigcap_{k \geq 3}(C_k(\g)+\g^k}\biggr)\biggr).
$$
Since $\g^2$ is contained in $\g^1$, 
we can apply again the $+$ $\cap$ distribution property to
$\g^1 \bigcap \biggl( C_2(\g)+
\g^2 \bigcap\biggl(\displaystyle{\bigcap_{k \geq 3}(C_k(\g)+\g^k}\biggr)\biggr)$, 
and get,
$$
\j(\g)=C_1(\g)+\g^1 \cap C_2(\g)
+\g^2 \bigcap\biggl(\displaystyle{\bigcap_{k \geq 3}(C_k(\g)+\g^k}\biggr)\biggr).
$$
Proceeding in this way, we get,
\begin{equation}\label{formula para j}
\j(\g)=C_1(\g)+\sum_{i=1}^{\ell-1}\left(C_{i+1}(\g)\cap \g^{i} \right)
+\g^{\ell} \bigcap \biggl(\bigcap_{k \geq 1} \left(C_{\ell+k}(\g)+\g^{\ell+k}  \right)  \biggr),
\end{equation}
which holds true for all $\ell \in \N$.
Since $\dim_{\F}(\g)$ is finite, there are natural numbers
$m_1$ and $m_2$ for which the descending central and
derived central series stabilize; that is,
$$
\g^0=\g \supset \g^1 \supset \cdots \supset \g^{m_1-1} \supset \g^{m_1}=\g^{m_1+ k}, 
$$
and,
$$
C_1(\g) \subset \cdots \subset C_{m_2-1}(\g) \subset C_{m_2}(\g)=C_{m_2+k}(\g)
$$
for all $k \in \N$, respectively.
If $\g$ is nilpotent, then $\g^{m_1}=\{0\}$ and $C_{m_2}(\g)=\g$. 
Let $m=\operatorname{max}\{m_1,m_2\}$. 
Taking $\ell=m$ in \eqref{formula para j} we get,
\begin{equation}\label{formula para j-2}
\j(\g)=C_1(\g)+\sum_{i=1}^{m-1}\left(C_{i+1}(\g)\cap \g^{i} \right)+\g^{m} \bigcap \biggl(\bigcap_{k \geq 1} \left(C_{m+k}(\g)+\g^{m+k}  \right)  \biggr).
\end{equation}
By the choice of $m$, we conclude that,
$$
C_{m}(\g)+\g^m=C_{m+1}(\g)+\g^{m+1}=\cdots=C_{m+k}(\g)+\g^{m+k},\,\,\forall k \in \N.
$$
Therefore,
$$
\bigcap_{k\geq 1} \biggl(C_{m+k}(\g)+\g^{m+k}\biggr)=C_{m+1}(\g)+\g^{m+1}=C_{m}(\g)+\g^m.
$$
Substituting this in \eqref{formula para j-2}, we obtain,
$$
\j(\g)=C_1(\g)+\sum_{i=1}^{m-1}\left(C_{i+1}(\g)\cap \g^{i} \right)
+\g^m \bigcap \left(C_m(\g)+\g^m \right).
$$
Now, applying the $+$ $\cap$ distribution property, we see that,
$$
\g^m \cap \left(C_{m}(\g)+\g^m \right)=C_m(\g) \cap \g^m +\g^m.
$$
Therefore,
$$
\j(\g)=C_1(\g)+\displaystyle{\sum_{k=1}^{m-1}C_{k+1}(\g) \cap \g^k}+\g^m.
$$
\medskip
{\bf (iii)} 
Let $\j^{\prime}(\g)$ be be defined by,
$$
\j^{\prime}(\g)=C_1(\g)+\sum_{k=1}^{m-1} C_{k+1}(\g) \cap \g^{k},
$$
so that $\j(\g)=\j^{\prime}(\g)+\g^m$. 
We claim that $\j^{\prime}(\g)$ is Abelian and that $[\j^{\prime}(\g),\g^m]=\{0\}$.
To see that $\j^{\prime}(\g)$ is Abelian, take
$1 \leq k \leq \ell\leq m-1$. 
Then $\g^k \supset \g^{\ell}$ and $C(\g^k) \subset C(\g^{\ell})$. 
Let $x \in C_{k+1}(\g) \cap \g^k$ and $y \in C_{\ell+1}(\g) \cap \g^{\ell}$. 
By {\bf Prop. \ref{proposicion auxiliar}}, we have 
$[x,y] \in [C_{k+1}(\g),\g^{\ell}] \subset [C(\g^k),\g^{\ell})] \subset [C(\g^{\ell}),\g^{\ell}]=\{0\}$. 
Therefore, $[\j^{\prime}(\g),\j^{\prime}(\g)]=\{0\}$.

\smallskip
To see that $[\j^{\prime}(\g),\g^m]=\{0\}$,
take $0 \le k \le m-1$. Since $\g^m \subset \g^{m-1}$,
{\bf Prop. \ref{proposicion auxiliar}} implies that, 
$$
[C_{k+1}(\g) \cap \g^k\!,\g^m] \!\subset\! [C_{m}(\g),\g^m] 
\!\subset\! [C(\g^{m-1}),\g^m] \!\subset\! [C(\g^m),\g^m]\!=\!\{0\}.
$$
Therefore, $[\j^{\prime}(\g),\g^m]=\{0\}$.
It now follows that $[\j(\g),\j(\g)]=[\g^m,\g^m]$,
and it is clear that $\j(\g)$ is Abelian if and only if $\g^m$ is Abelian.
If $\ell<m$, then $\g^{\ell} \supset \g^m$, which implies that $\g^m$ is also Abelian.
The choice of $m$ implies that, $\g^{\ell}=\g^m$ if $\ell\ge m$.
Therefore, $\g^m$ is Abelian whenver $\g^\ell$ is.

\medskip
{\bf (iv)} 
Choose again $m=\operatorname{max}\{m_1,m_2\}$ as in 
{\bf (ii)} above, where $m_1$ and $m_2$ are the natural numbers 
for which the descending central and derived central series respectively stabilize.
Observe that, 
\begin{equation}\label{definicion de ide}
\ide(\g)=\sum_{k=1}^m C_k(\g) \cap \g^k
=\sum_{k=1}^{m-1}C_k(\g) \cap \g^k+C_m(\g) \cap \g^m,
\end{equation}
Let $B:\g \times \g \to \F$ be an invariant metric and assume that $\j(\g)$ is Abelian. 
By {\bf Prop. \ref{proposicion auxiliar}}, we know that 
$(\g^m)^{\perp}=C_m(\g)=C(\g^{m-1})$. 
Since $\g^m=\g^{m+1}=[\g,\g^m]$ is Abelian and $B$ is invariant, we have,
$$
B(\g^m,\g^m)=B([\g,\g^m],\g^m) \subset B(\g,[\g^m,\g^m])=\{0\}.
$$
Thus, $\g^m$ is isotropic. Therefore, 
$\g^m \subset (\g^m)^{\perp}=C_m(\g)$
and consequently, $\g^m \cap C_m(\g)=\g^m$. 
Thus, by \eqref{definicion de ide}, we obtain:
$$
\ide(\g)=\sum_{k=1}^m C_k(\g) \cap \g^k=\sum_{k=1}^{m-1}C_k(\g) \cap \g^k+\g^m.
$$
Therefore,
$$
\aligned
\,[\g,\j(\g)] & \subset  \sum_{k=1}^{m-1}[\g,C_{k+1}(\g) \cap \g^k]+[\g,\g^m]\\
\,& \subset \sum_{k=1}^{m-1}C_{k+1}(\g) \cap \g^{k+1}+\g^m \subset \ide(\g).
\endaligned
$$
Now, we shall use {\it the $+$ $\cap$ $\perp$ properties\/} satisfied for any
pair of vector subspaces $V$ and $W$ of a quadratic $\g$; namely, 
$(V+W)^{\perp}=V^{\perp} \cap W^{\perp}$ and $(V\cap W)^{\perp}=V^{\perp}+W^{\perp}$,
respectively. Therefore,
$$
\aligned
\ide(\g)^{\perp}&=\biggl(\sum_{k \in \N}C_k(\g) 
\cap \g^k \biggr)^{\perp}=\bigcap_{k \in \N}\biggl(C_k(\g) \cap \g^k \biggr)^{\perp}\\
\,&=\bigcap_{k \in \N}\biggl(C_k(\g)^{\perp}
+(\g^k)^{\perp}\biggr)=\bigcap_{k \in \N}\biggl(\g^k + C_k(\g)\biggr)=\j(\g).
\endaligned
$$

\smallskip
Now, recall from {\bf (iii)} that $\j^{\prime}(\g)$ is defined in terms of
the canonical ideals $C_1(\g)$ and $C_{k+1}(\g)\cap\g^k$ so that
$\j(\g)=\j^{\prime}(\g)+\g^m$. Whence, $\j^{\prime}(\g)$ also satisfies
$[\g,\j^{\prime}(\g)] \subset \ide(\g)$. 
The reason why we consider the ideal $\j(\g)$ instead $\j^{\prime}(\g)$
for our work is because that in the presence of an invariant metric in $\g$,
$\j(\g)$ is the orthogonal complement of $\ide(\g)$.

\smallskip
{\bf (v)} 
Let $\g$ be a solvable non-Abelian Lie algebra 
admitting an invariant metric. 
Suppose that there exists $\ell \in \N$ such that $\g^{\ell}$ is Abelian. 
Then $\g \neq \g^{\ell}$. If $C(\g)=\{0\}$, then $\g=[\g,\g]=\g^1=\cdots =\g^{\ell}$, 
which is a contradiction. 
Therefore, $C(\g) \neq \{0\}$. This guarantees that $\ide(\g) \neq \{0\} \neq \j(\g)$.
\end{proof}

\smallskip
\begin{Remark}\label{aclaracion1}{\rm
It is worth noting that the definition of $\ide(\g)$ and $\j(\g)$ 
given in {\bf Lemma \ref{lema auxiliar 2}} coincides
with the characterization of the two ideals
given in \cite{Kath} when the Lie algebra $\g$ is nilpotent.
As a matter of fact, the ideals appearing in \cite{Kath}
are constructed using the  $+$ $\cap$ distribution property with
a family of ideals $R_k(\g)$ of $\g$, satisfying,
$$
\g=R_0(\g) \supset R_1(\g) \supset \cdots \supset R_{\ell}(\g)=\{0\},
$$
(see {\bf Def. 4.1} in \cite{Kath}). 
It is not difficult to prove that for a nilpotent Lie algebra $\g$, $R_k(\g)=\g^k$ for all $k$. 
Thus, in the solvable non-nilpotent case, we obviously have,
$R_k(\g) \neq \g^k \neq \{0\}$ for all $k\in\N$. 
In particular, {\it the ideals $\ide(\g)$ and $\j(\g)$ given in\/} 
{\bf Lemma \ref{lema auxiliar 2}} 
{\it do differ in the solvable non-nilpotent case from those of\/} \cite{Kath}.
At this point it is convenient to make the following
{\bf assumption:} {\it From now on, 
all solvable Lie algebras considered in this work
will not be Abelian\/;} 
that is, we shall assume that if $\g$ is solvable 
is because $\g^{(\ell)}=\{0\}$, for some $\ell>1$.
It will also be assumed that that there is some $k\in\N$
such that $[\g^k, \g^k] =\{0\}$; {\it that is, $\g$ has an
Abelian descending central ideal.}
}
\end{Remark}

\smallskip
As a corollary, the ideals $\ide(\g)$ and $\j(\g)$ of {\bf Lemma \ref{lema auxiliar 2}}
now yield the following vector space decomposition of a solvable Lie algebra $\g$ 
under the assumption just stated.

\smallskip
\begin{Cor}\label{cor descomposicion}{\sl
Let $\g$ be a solvable Lie algebra with one Abelian descending central ideal. 
Let $\ide(\g)$ and $\j(\g)$ the Abelian characteristic ideals of {\bf Lemma \ref{lema auxiliar 2}}. 
Then, there are complementary subspaces $\h$ and $\a$ such that,
\begin{itemize}

\item[(i)] $\g=\h \oplus \j(\g)$, with $\{0\} \neq \h \simeq \g/\j(\g)$.

\item[(ii)] $\j(\g)=\a \oplus \ide(\g)$ and $\g=\h \oplus \a \oplus \ide(\g)$.

\end{itemize}
Furthermore, if $\g$ admits an invariant metric, 
then $\ide(\g)^{\perp}=\j(\g)$ and $\h$ and $\a$ can be chosen
in such a way that $\h$ is isotropic and $\a^{\perp}=\h \oplus \ide(\g)$.
}
\end{Cor}

\smallskip
\begin{Remark}\label{aclaracion2}
{\rm 
An identical vector space decomposition of $\g$ is given in \cite{Kath}
for the two ideals defined there. Differences between their results 
and ours appear for the solvable non-nilpotent $\g$'s we deal with hereby.
To start with, the Lie brackets are different in the non-coincident cases.
For example, $\h$ acts semisimply in their quotient $\j(\g)/\ide(\g)\simeq \a$
(see {\bf Lemma 4.2.(b)} in \cite{Kath}) and this action is trivial if $\g$ is nilpotent,
but it is not when $\g$ is solvable but not nilpotent.
On the other hand, under our hypotheses, $\h$ {\it always acts trivially on\/} $\j(\g)/\ide(\g)$,
even for non-nilpotent $\g$, but this is due to the fact that $\j(\g)$ is Abelian,
which cannot be guaranteed for the corresponding ideal obtained in \cite{Kath}.
}
\end{Remark}

\smallskip
\section{Solvable Quadratic Lie Algebras with one Abelian descending central ideal}

\smallskip
Let $\g$ be a solvable quadratic Lie algebra with 
Lie bracket $[\,\cdot\,,\,\cdot\,]:\g\times\g\to\g$,
invariant metric $B:\g\times\g\to\F$ 
and assume that there is some $\ell \in \N$ such that $\g^{\ell}$ is Abelian.
Let $\ide:=\ide(\g)$ and $\j:=\j(\g)$
be the characteristic ideals defined in {\bf Lemma \ref{lema auxiliar 2}}.
It follows from {\bf Lemma \ref{lema auxiliar 2}.(v)} that 
$\g\ne [\g,\g]$, and in fact, $C(\g)=[\g,\g]^\perp$.
We also know from {\bf Lemma \ref{lema auxiliar 2}} that $\ide^\perp=\j \supset \ide$. 
Thus, $\ide$ is a canonically defined isotropic ideal of $\g$.
By Witt decomposition, $\g=\h\oplus\a\oplus\ide$, with $\h$ isotropic
and $\ide^\perp=\a\oplus\ide$. In particular,
$\a$ is non-degenerate and $\a^\perp=\h\oplus\ide$.

\smallskip
Observe that $B\vert_{\h\times\ide}:\h\times\ide\to\F$ cannot degenerate
and one may identify $\ide$ with $\h^*$ (or else, identify $\h$ with $\ide^*$,
which at the end is a matter of convenience depending on the
data one wants to start with). The identification of $\ide$ with $\h^*$
goes as follows: 
Start with $\theta\in \ide$ and consider the map $\theta\to\theta^\flat\in\h^*$
defined by $\theta^\flat(x)=B(\theta,x)$, for any $x\in\h$. 
This map is equivariant
and intertwines the representation $\rho$ of $\h$ in $\ide$ with the 
coadjoint representation $\ad_{\h}^*$ of $\h$ in $\h^*$.
Indeed, since, $\ide^\perp=\a\oplus\ide$, 
$$
\aligned
\left(\rho(x)(\theta)\right)^\flat(y)&=
B(\rho(x)(\theta),y) = 
B([x,\theta],y)
\\
&=-B(\theta,[x,y])
=-B(\theta,[x,y]_{\h})
\\
&=-(\theta^\flat)([x,y]_{\h})
= -(\theta^\flat)\circ\ad_{\h}(x)\,(y)
\\
& =\left(\ad_{\h}^*(x)\,\theta^\flat\right)(y).
\endaligned
$$
{\bf Convention.} Having assumed that $\g=\h\oplus\a\oplus\ide$
is quadratic with invariant symmetric 
bilinear form $B$, we shall, from now on, 
identify $\ide$ with $\h^*$ 
and will use the coadjoint representation in $\h^*$ instead of $\rho$.

\smallskip
We now want to look at the decomposition of the representation $R:\h\to\gl(\a\oplus\h^*)$
into the map $\varphi:\h\to\Hom_{\F}(\a,\h^*)$,
and the coadjoint representation $\ad_{\h}^*:\h\to\gl(\h^*)$.
Observe that for any $y\in\h$,
$$
\aligned
B(\varphi(x)(v),y) & = B([x,v],y)=-B(v,[x,y])
\\
&=-B(v,[x,y]_{\h}+\lambda(x,y)+\mu(x,y))
\\
&=-B(v,\lambda(x,y)).
\endaligned
$$ 
This says that 
$\a^*\ni -\lambda(x,y)^\flat=B(\varphi(x)(\,\cdot\,),y)$
is related to the map,
$\varphi(x)^*:\h\to \a^*$,
through the following:
\begin{equation}\label{lambda-varphi}
-\lambda(x,y)^\flat(v) =(\varphi(x)^*(y))(v)= (\varphi(x)(v))(y).
\end{equation}
In other words, when the Lie algebra is quadratic,
$\lambda$ can be built up from $\varphi$ or viceversa,
and only one of the two is needed in specifying
the Abelian extension of {\bf Cor. \ref{corolario ext abelianas}}.

\smallskip
Finally, observe that for any triple $x$, $y$ and $z$ in $\h$, we get:
\begin{equation}\label{mu ciclicla}
\begin{array}{rl}
\mu(x,y)^\flat(z)&=B(\mu(x,y),z)=B([x,y],z)
\\
&=B(x,[y,z])
=B(x,\mu(y,z))=\mu^\flat(y,z)(x). 
\end{array}
\end{equation}
Therefore, under the identification of $\ide$ with $\h^*$,
the $2$-cochain $\mu$ with values in $\h^*$, has the cyclic property,
$$
\mu(x,y)(z)=\mu(y,z)(x).
$$
It is worth noting that 
cyclic cochains are found in several related contexts; 
see for example \cite{BajoBenayadiMedina} or \cite{Bordemann}. 
The cyclic property appears here due to the existence of an invariant metric.

\smallskip
Now, observe that  \eqref{lambda-varphi} says that in the presence
of an invariant metric $B$ on $\g$, $\varphi$ and $\lambda$ can be 
determined one from each other.
Assume for the moment that $\g$ does not have an invariant metric
but suppose that we are given a non-degenerate symmetric bilinear
form $B_{\a}:\a \times \a \to \F$ in $\a$. 
The non-degeneracy of $B_{\a}$ implies that there exists a bilinear map 
$\lambda_{\varphi}:\h \times \h \to \a$, such that
for all $x$ and $y$ in $\h$,
\begin{equation}
\label{ecuacion auxiliar lambda}
\varphi^{*}(x)(y)(v)=-B_{\a}\left(\lambda_{\varphi}(x,y),v\right),\quad 
\text{for all\ }v \in \a.
\end{equation} 
Under the guide given by \eqref{lambda-varphi},
a necessary condition for $\g$ to admit an invariant metric is
that $\lambda_{\varphi}$ coincides with $\lambda$. 
In principle, $\lambda_{\varphi}$ and $\lambda$ may not be the same,
but if the difference $\lambda_{\varphi}-\lambda$ is a coboundary,
we may use {\bf Cor. \ref{criterio cohomologico}} to obtain 
a useful criterion to determine whether a solvable Lie algebra 
having an Abelian descending central ideal admits an invariant metric or not. 
Thus, we may now rephrase the main result on Abelian extensions
for the case in which $\g=\h(\lambda,\mu,\varphi,\rho)$ admits 
an invariant metric $B$ as follows:

\smallskip
\begin{Theorem}\label{nilp cuad}{\sl
{\rm {\bf (1)}}
Let $(\h,[\cdot,\cdot]_{\h})$ be a finite dimensional Lie algebra. 
Let $\h$ act on the dual vector space $\h^*$ through the coadjoint representation
$\ad_{\h}^*:\h\to\gl(\h^*)$. Let  $\a$ be a finite dimensional trivial $\h$-module. 
Let $\overline{\ad_{\h}^{*}}:\h \to \gl(\Hom_{\F}(\a,\h^{*}))$ be the 
tensor product representation on $\h^*\otimes\a^*\simeq \Hom_{\F}(\a,\h^{*})$
so that $\overline{\ad_{\h}^{*}}(x)(T)=\ad_{\h}^{*}(x) \circ T$, for all $x \in \h$ and 
all $T \in \Hom_{\F}(\a,\h^{*})$. 
Let $\varphi:\h\to\Hom_{\F}(\a,\h^*)$ be a 1-cocycle 
with coefficients in the representation $\overline{\ad_{\h}^{*}}$
and let
$e_{\varphi}:C(\h,\a) \to C(\h,\h^{*})$, be the degree-one map
of cochain complexes defined in \eqref{operador e}. 
Finally, let 
$\d_{\varphi}:C(\h,\a)\oplus C(\h,\h^{*}) \to C(\h,\a)\oplus C(\h,\h^{*})$
be the differential map,
$$
\d_{\varphi}
\begin{pmatrix}
\lambda \\ \mu
\end{pmatrix}
=
\begin{pmatrix}
\d\lambda \\
\d\mu + e_{\varphi}(\lambda)
\end{pmatrix}.
$$
Given a $2$-cochain $\left(\begin{smallmatrix}\lambda\\ \mu\end{smallmatrix}\right)$ 
define the following skew-symmetric bilinear map
$[\,\cdot\,,\,\cdot\,]_{\{\lambda,\mu\}}:\g \times \g \to \g$:
$$
\aligned
\forall\, x,y \in \h, & \quad \ [x,y]_{\{\lambda,\mu\}}=[x,y]_{\h}+\lambda(x,y)+\mu(x,y),\\
\forall\, x \in \h,\forall\, v \in \a, & \quad \  [x,v]_{\{\lambda,\mu\}}=\varphi(x)(v),\\
\forall\,  x \in \h,\forall\, \alpha \in \h^{*}, & \quad \  [x,\alpha]_{\{\lambda,\mu\}}=\ad_{\h}^{*}(x)(\alpha),\\
\forall\, u,v \in \a,\,\forall\, \alpha,\beta \in \h^{*}, & \quad \ [u+\alpha,v+\beta]_{\{\lambda,\mu\}}=0.
\endaligned
$$
Then $[\,\cdot\,,\,\cdot\,]_{\{\lambda,\mu\}}$ defines a Lie bracket in $\g$
if and only if $\left(\begin{smallmatrix}\lambda\\ \mu\end{smallmatrix}\right) \in \Ker D_{\varphi}$.

\smallskip
{\rm {\bf (2)}}
Let $B_{\a}:\a \times \a \to \F$ be a symmetric and non-degenerate bilinear form on $\a$.
Let $B:\g \times \g \to \F$ be the bilinear form defined by
\begin{equation}\label{definicion de metrica}
B(x+u+\alpha,y+v+\beta)=\alpha(y)+\beta(x)+B_{\a}(u,v),
\end{equation} 
for all $x,y \in \h$, $u,v \in \a$ and $\alpha, \beta \in \h^{*}$. Let $\varphi^{*}:\h \to \Hom_{\F}(\h,\a^{*})$ be the linear map defined by
$(\varphi^*(x)(y))(v)= (\varphi(x)(v))(y)$,
for all $x,y \in \h$ and $v \in \a$.
Let $\lambda_{\varphi}:\h \times \h \to \F$ be the bilinear satisfying:
$$
\varphi^{*}(x)(y)=-B_{\a}(\lambda_{\varphi}(x,y),\,\cdot\,),\quad \mbox{ for all }x,y \in \h,
$$
just as in \eqref{ecuacion auxiliar lambda}. 
If $(\g,[\cdot,\cdot]_{\{\lambda,\mu\}},B)$ is a quadratic Lie algebra 
then the following conditions are satisfied: 
\begin{itemize}

\item[{\bf (b1)}] $\mu(x,y)(z)=\mu(y,z)(x)$ for all $x,y,z \in \h$.

\item[{\bf (b2)}] There exists $\tau \in \Ker(e_{\varphi})$ 
such that $\lambda_{\varphi}=\lambda+\d\tau$.

\end{itemize}
Reciprocally, if {\bf (b1)} and {\bf (b2)} hold true, 
then $(\g,[\cdot,\cdot]_{\{\lambda,\mu \}})$ admits an invariant metric. 
Furthermore, any solvable quadratic Lie algebra $\g$ having
an Abelian descending central ideal, has the structure $\h\oplus\a\oplus\h^*$ 
where,
$\h^{*} \simeq \ide(\g)$, $\a \oplus \h^{*}\simeq \j(\g)$
and the pair $\{\lambda,\mu\}$ is defined through the components
of the Lie bracket in $\g$ of any two elements in $\h$
along $\a$ and $\h^*$, respectively.
}
\end{Theorem}

\smallskip
\begin{proof}
{\bf (1) } The statement gives necessary and sufficient conditions for 
$[\,\cdot\,,\,\cdot\,]_{\{\lambda,\mu\}}$ to be a Lie bracket on $\g$.
The general conditions have already been given
in {\bf Cor. \ref{corolario ext abelianas}}.
We are applying it here to 
the special case in which $\ide(\g)\simeq\h^*$
and $\rho=\ad^*_{\h}$ are fixed,
and  $\varphi$ is obtained from $\lambda$ (or viceversa) due to
the existence of an invariant metric in $\g$ as shown in \eqref{lambda-varphi}.

\smallskip
First suppose that the bilinear $B$ given in \eqref{definicion de metrica}
is indeed an invariant metric in $\g$. 
Then, by \eqref{lambda-varphi} and \eqref{mu ciclicla}
imply that {\bf (b1)} and {\bf (b2)} hold true with $\tau=0$.

\smallskip
Conversely, assume {\bf (b1)} and {\bf (b2)} are satisfied. 
It follows from equations \eqref{criterio cohomologico 1} in
{\bf Cor. \ref{criterio cohomologico}} that the Lie algebra 
$(\g,[\cdot,\cdot]_{\{\lambda,\mu \}})=\h(\lambda,\mu,\varphi,\ad^{*}_{\h})$
is isomorphic to $(\g,[\cdot,\cdot]_{\{\lambda+\d\!\tau,\,\mu \}})
=\h(\lambda+\d\tau,\mu,\varphi,\ad^{*}_{\h})$. 
The vector space decomposition of the underlying $\g$ is the same 
in both cases; namely,
$\g=\h \oplus \a \oplus \h^{*}$. 
Let $B$ be the symmetric, non-degenerate 
bilinear form defined in $\h \oplus \a \oplus \h^{*}$ 
by \eqref{definicion de metrica}. 
We claim that $B$ is an invariant metric on 
$\h(\lambda+\d\tau,\mu,\varphi,\ad^{*}_{\h})$. 
Indeed, the Lie bracket in $\h(\lambda+\d\tau,\mu,\varphi,\ad^{*}_{\h})$ is given by:
$$
\aligned
\,[x,y]_{\{\lambda+\d\!\tau,\,\mu \}}&=[x,y]_{\h}+(\lambda+\d\tau)(x,y)+\mu(x,y),\\
[x,v]_{\{\lambda+\d\!\tau,\,\mu \}}&=\varphi(x)(v),\\
[x,\alpha]_{\{\lambda+\d\!\tau,\,\mu \}}&=\ad_{\h}(x)^{*}(\alpha),
\endaligned
$$
for all $x,y \in \h$, $v \in \a$ and for all $\alpha \in \h^{*}$. 
Notice that the cyclic condition for $\mu$ given in \eqref{mu ciclicla} is satisifed. 
In addition,
$$
\varphi^{*}(x)(y)(v)=-B_{\a}(\lambda_{\varphi}(x,y),v)=-B_{\a}((\lambda+\d\tau)(x,y),v),
$$
for all $x,y \in \h$ and for all $v \in \a$.
Therefore, \eqref{lambda-varphi} is satisfied for the pair $(\lambda+\d\tau,\varphi)$, 
thus proving that $(\g,[\cdot,\cdot]_{\{\lambda+\d\!\tau,\,\mu \}},B)$
is a quadratic Lie algebra. Consequently,
$(\g,[\cdot,\cdot]_{\{\lambda,\mu\}})$ admits an invariant metric.

\smallskip
Now, we want to show that {\it any\/} solvable
quadratic Lie algebra $\g$ having an Abelian descending central ideal
has its Lie bracket and its invariant metric given as in the statement.
Let $\g$ be such an algebra.
By {\bf Lemma \ref{lema auxiliar 2}.(v)}, 
we know that $\{0\} \neq C(\g) \subset \j(\g)$,
which implies that $\j(\g) \neq \{0\}$. Also, {\bf Lemma \ref{lema auxiliar 2}.(iv)}
says that $\j(\g)^{\perp}\!=\ide(\g)$. 
If $\ide(\g)=\{0\}$, the non-degeneracy of the invariant metric implies that $\g=\j(\g)$.
But, {\bf Lemma \ref{lema auxiliar 2}.(iii)}, says that $\j(\g)$ is Abelian, contrary to the hypothesis on $\g$.
Therefore, $\ide(\g) \neq \{0\}$. Let $\a$ and $\h$ be two subspaces of $\g$ such that
$\j(\g)=\a \oplus \ide(\g)$, and $\g=\h \oplus \j(\g)$. Then $\g=\h \oplus \a \oplus \ide(\g)$. 
Now apply {\bf Cor. \ref{corolario ext abelianas}}
to conclude that the Lie bracket on $\g$
is precisely the one given in the statement. 
Furthermore, the invariance and the non-degeneracy of the metric
also imply that the 2-cocycle
$\Lambda\leftrightarrow \left(\begin{smallmatrix}\lambda \\ \mu\end{smallmatrix}\right)$
and the 1-cocycle $\varphi$, satisfy the conditions given in {\bf (b1)} and {\bf (b2)}.
\end{proof}

\smallskip
\begin{Example}\label{ejemploX}{\rm
We shall now give an example of a solvable non-nilpotent
quadratic Lie algebra $\g$ having an Abelian descending central ideal
and we shall determine its canonical ideals $\ide(\g)$ and $\j(\g)$.

\smallskip
Let $V$ be a finite dimensional vector space over $\F$.
Let $B_V:V \times V \to \F$ be a non-degenerate, symmetric 
bilinear form on $V$. Let $D \in \End_{\F}(V)$ be such that 
$B_V(D(v),w)=-B_V(v,D(w))$, for all $v,w \in V$.
Finally, let $c$ be defined in such a way that $\g=\F D \oplus V \oplus \F c$
becomes a $(\operatorname{dim}V+2)$-dimensional vector space over $\F$
with Lie bracket,
\begin{equation}\label{ejemplo1}
[\alpha D+v+\beta c,\alpha^{\prime}D+w+
\beta^{\prime}c]=\alpha D(w)-\alpha^{\prime}D(v)+B(D(v),w)c,
\end{equation} 
for all $v,w\in V$ and for all $\alpha,\alpha^{\prime},\beta,\beta^{\prime}\in\F$. 
Clearly,
$$
C(\g)=\Ker(D) \oplus \F\,c,\qquad
\g^k =\Im(D^k) \oplus \F\,c,\quad\text{for all\ }k\in\N.
$$
Since $c \in C(\g)$, it is clear that $\g$ is a nilpotent Lie algebra if and only if $D$ is a 
nilpotent linear map on $V$. Moreover, if $D$ is non-nilpotent, then,
$$
C_k(\g)=\Ker(D^k) \oplus \F \,c,\quad\text{for all\ }k\in\N.
$$
Indeed, suppose that $C_k(\g)=\Ker(D^k) \oplus \F\,c$ and proceed by induction on $k$. 
Let $\alpha D+v+\beta c \in C_{k+1}(\g)$
and choose $w\in V$ such that $D^{k+1}(w) \neq 0$. 
The definition of $C_{k+1}(\g)$, leads to, 
$$
[\,\alpha D+v+\beta c\,,\,w\,] \in C_k(\g)=\Ker(D^k) \oplus \F\,c.
$$
Then, $\alpha\,D(w) \in \Ker(D^k)$, which implies that $\alpha\,D^{k+1}(w)=0$.
Therefore, $\alpha=0$.
On the other hand $[\,D\,,\,v+\beta\,c\,] \in C_k(\g)$ implies
$D(v) \in \Ker(D^k)$. Therefore, $v \in \Ker(D^{k+1})$.
This proves that $C_{k+1}(\g) \subset \Ker(D^{k+1}) \oplus \F\,c$. 
We shall now prove that $C_{k+1}(\g) \supset \Ker(D^{k+1}) \oplus \F\,c$.
It is clear that $\F\,c \subset C_{k+1}(\g)$. Let $v \in \Ker(D^{k+1})$. Then,
$$
[\alpha^{\prime}D+w+\beta^{\prime},v]
=\alpha^{\prime}D(v)+B_V(D(w),v)\,c \in \Ker(D^k) \oplus \F\,c=C_{k}(\g),
$$
for all $\alpha^{\prime},\beta^{\prime} \in \F$ and $w \in V$. 
This amounts to say that $v \in C_{k+1}(\g)$ which proves that 
$\Ker(D^{k+1}) \oplus \F\,c \subset C_{k+1}(\g)$.  
Therefore, $C_{k+1}(\g)=\Ker(D^{k+1}) \oplus \F\,c$. 

\smallskip
Now assume that $D$ is nilpotent, with $D^{m}=0$ but $D^{m-1} \neq 0$.
Just as in the non-nilpotent case, we have 
$C_k(\g)=\Ker(D^k) \oplus \F\,c$, for all $1 \leq k \leq m-1$. 
Observe that for $k=m$, 
$\Im D \subset \Ker D^{m-1}$. Then, it is not difficult to prove that $C_{m}(\g)=\g$.

\smallskip
Now consider the Lie algebra $\g=\F\,D \oplus V \oplus \F\,c$ 
with Lie bracket defined by \eqref{ejemplo1}. 
Let $m$ be the positive integer for which the descending central series 
and the derived central series both stabilize.
Using {\bf Lemma \ref{lema auxiliar 2}.(ii)} we conclude that,
$$
\aligned
\ide(\g)&=\left(\sum_{k=1}^m \Ker(D^k) \cap \Im(D^k) \right) \oplus \F\,c,\,\mbox{ and }\\
\j(\g)&=\left(\sum_{k=0}^{m-1}\Ker(D^{k+1}) \cap \Im(D^k)+\Im(D^m)\right) \oplus \F\,c.
\endaligned
$$

Now, we define a symmetric bilinear form $B:\g \times \g \to \F$, by:
$$
B\left(\alpha D+v+\beta c,\alpha^{\prime}+w+\beta^{\prime}c\right)
=B_V\left(v,w \right)+\alpha \beta^{\prime}+\alpha^{\prime} \beta,
$$
for all $\alpha$, $\alpha^\prime$, $\beta$ and $\beta^\prime$ in $\F$ and
any $v$ and $w$ in $V$.
Then, $B$ is an invariant metric in the solvable Lie algebra $\g$. 
In fact, $\g$ is the {\it double extension of $V$ by the pair 
$(D,\omega)$\/,} where $\omega:V \times V \to \F$ is the
skew-symmetric bilinear form on $V$ defined by 
$\omega(u,v)=B \left(D(u),v \right)$, for all $u,v \in V$ (see \cite{MCGS}, {\bf Thm. 3.2 }).
}

\smallskip
{\rm
With these preliminaries in mind, we may now produce a concrete
example of a solvable quadratic Lie algebra having an Abelian 
descending central ideal.
Indeed, let $V$ be a 3-dimensional space over $\F$.
Let $\{v_1,v_2,v_3\}$ be a basis for $V$ and let 
$B_V:V \times V \to \F$ be the bilinear form defined by,
$$
\aligned
& B_V(v_1,v_1)=B_V(v_2,v_3)=1,\\
& B_V(v_1,v_2)=B_V(v_1,v_3)=B_V(v_2,v_2)=B_V(v_3,v_3)=0.
\endaligned
$$
In particular $\{v_2,v_3\}$ is a hyperbolic pair for $B_V$.
Let $D:V \to V$ be the linear map defined by 
$D(v_1)=0$, $D(v_2)=v_3$ and $D(v_3)=-v_2$.
In the vector space $\g=\F\,D \oplus V \oplus \F\,c$ 
consider the quadratic Lie algebra structure of 
the double extension of $V$ by $(D,\omega)$ as in \cite{MCGS}.
A straightforward calculation shows that,
$$
C(\g)=\F\,v_1 \oplus \F\,c\quad \text{and}\quad \g^1=\Span_{\F}\{v_2,v_3\} \oplus \F\,c.
$$
Now, $\g^1$ is Abelian because $\{v_2,v_3\}$ is a hyperbolic pair.
Clearly, $D$ is invertible in the subspace $\Span_{\F}\{v_2,v_3\}$, 
which implies that $\g$ is a non-nilpotent Lie algebra. 
Then, it is not difficult to prove that,
$$
\ide(\g)=\F v_1 \oplus \F\,c,\quad \mbox{ and }\quad \j(\g)=V \oplus \F\,c.
$$
We refer the reader to \cite{MCGS} for
a more general and exhaustive study of this type of quadratic Lie algebras.
}
\end{Example}

\smallskip
\section{Classification of Lie Algebras $\g=\h\oplus\a\oplus\h^*$
where $\h$ is the 3-dimensional Heisenberg Algebra and $\a=\F^r$}

\smallskip
Start by observing that a nilpotent Lie algebra
is a Lie algebra that has at least one Abelian descending central ideal. 
In this section we aim to classify, up to isomorphism, all the nilpotent
Lie algebras $\g$ for which
$\h=\g/\j(\g)$ is the $3$-dimensional Heisenberg Lie algebra and 
$\j(\g)/\ide(\g)\simeq\Bbb F^r$, with $r \geq 3$.
Said classification will illustrate how 
{\bf Cor. \ref{corolario ext abelianas}} and {\bf Thm. \ref{nilp cuad}} work.

\smallskip
Let $\h=\Span_{\F}\{x_1,x_2,x_3\}$ be the 3-dimensional Heisenberg Lie algebra,
with  $C(\h)=\Span_{\F}\{x_3\}$, and $[x_1,x_2]=x_3$. 
Let $\a=\F^r$ be the trivial $\h$-module with $r \geq 3$.
Let $\rho=\ad^{*}:\h \to \gl(\h^{*})$ be the coadjoint representation of $\h$.
Let $\varphi$ be a $1$ cocycle in $C(\h;\Hom(\F^r,\h^*))$
associated to the representation $\overline{\ad^{*}}:\h \to \gl(\Hom(\F^r,\h^{*}))$ 
given by $T\mapsto \overline{\ad^{*}}(x)(T)=\ad^*(x)\,\circ\,T$.
Let $\h(\lambda,\mu,\varphi)$ be the Lie algebra defined on 
$\g=\h\oplus\F^r\oplus\h^*$ by the data 
$(\lambda,\mu,\varphi)$ as in {\bf Cor. \ref{corolario ext abelianas}} 
with $\rho=\ad^*$ fixed, $\ide(\g)=\h^*$ and $\j(\g)=\F^r\oplus\h^*$.

\smallskip
Let $\{v_1,\cdots,v_r\}$ be a basis for $\F^r$ and let 
$\{\omega^1,\cdots,\omega^r\}$ be its dual basis.
Let $\{\theta^1,\theta^2,\theta^3\}$ be the basis of $\h^*$ dual to $\{x_1,x_2,x_3\}$. 
Since $\varphi(x) \in \Hom_{\F}(\F^3,\h^{*})\simeq\h^{*} \otimes (\F^3)^{*}$,
for all $x\in \h$, we may write,
$$
\varphi(x_j)=\sum_{i=1}^3\sum_{\ell=1}^r\,{\varphi^j}_{i \ell}\,\theta^i
\otimes \omega^{\ell}, \quad 1 \leq j \leq 3.
$$
Thus, each $\varphi(x_j):\F^r \to \h^{*}$
gets associated to the $3 \times r$ matrix,
$$
\varphi(x_j)\ \leftrightarrow\ 
\varphi^j=\begin{pmatrix}
{\varphi^j}_{11} & \varphi^j_{12} & \cdots & {\varphi^j}_{1r} \\
{\varphi^j}_{21} & \varphi^j_{22} & \cdots & {\varphi^j}_{2r} \\
{\varphi^j}_{31} & \varphi^j_{32} & \cdots & {\varphi^j}_{3r}
\end{pmatrix},
\quad 1 \leq j \leq 3.
$$
The Lie bracket of the Heisenberg Lie algebra $\h$, gives the following relations:
$$
\aligned
\ad^{*}(x_1)(\theta^1)=\ad^{*}(x_1)(\theta^2)&=0,\\
\ad^{*}(x_2)(\theta^1)=\ad^{*}(x_2)(\theta^2)&=0,
\endaligned
$$
$$
\ad^{*}(x_1)(\theta^3)=-\theta^2,
\quad 
\ad^{*}(x_2)(\theta^3)=\theta^1, 
\quad 
\ad^{*}(x_3)=0.
$$
Since, $\varphi \in\Hom_{\F}(\h,\Hom_{\F}(\F^3,\h^{*}))$ is a $1$-cocycle, it satisfies,
\begin{equation}\label{condicion de 1 cociclo}
\varphi([x_i,x_j])=\ad^{*}(x_i)\circ\,\varphi(x_j)-\ad^{*}(x_j)\circ\,\varphi(x_i),
\quad 1\le i,j\le 3.
\end{equation}
A straightforward computation shows that \eqref{condicion de 1 cociclo} yields,
$$
\varphi^3=
-\begin{pmatrix}
{\varphi^1}_{31} & \cdots & {\varphi^1}_{3r} \\
{\varphi^2}_{31} & \cdots & {\varphi^2}_{3r} \\
0 & \cdots & 0
\end{pmatrix};
$$
that is,
$$
{\varphi^3}_{1\ell}= -{\varphi^1}_{3\ell},\quad
{\varphi^3}_{2\ell}= -{\varphi^2}_{3\ell},\quad
{\varphi^3}_{3\ell}=0,\quad 1 \leq \ell \leq r.
$$
Thus, the entries of $\varphi^3$ are
completely determined by the entries of
$\varphi^1$ and $\varphi^2$.
We now write the $2$-cocycle $\lambda:\h \times \h \to \F^3$, in the form,
$$
\aligned
\lambda(x_1,x_2)&=\lambda_{13}\,v_1+\lambda_{23}\,v_2+\cdots+\lambda_{r3}\,v_r,\\
\lambda(x_2,x_3)&=\lambda_{11}\,v_1+\lambda_{21}\,v_2+\cdots+\lambda_{r1}\,v_r,\\
\lambda(x_3,x_1)&=\lambda_{12}\,v_1+\lambda_{22}\,v_2+\cdots+\lambda_{r2}\,v_r,
\endaligned
$$
with $\lambda_{\ell j} \in \F$; $1\le \ell\le r$, $1\le j \le 3$. 
Thus, $\lambda$ gets identified with the $r\times 3$ matrix,
\begin{equation}\label{matriz-lambda}
\lambda \ \leftrightarrow\ 
\begin{pmatrix}
\lambda_{11}&\lambda_{12}&\lambda_{13}\\
\lambda_{21}&\lambda_{22}&\lambda_{23} \\ 
\vdots & \vdots & \vdots \\
\lambda_{r1} & \lambda_{r2} & \lambda_{r3}
\end{pmatrix}.
\end{equation}
Similarly, $\mu:\h \times \h \to \h^{*}$ can be written in the form,
$$
\aligned
\mu(x_1,x_2)&=\mu_{13}\,\theta^1+\mu_{23}\,\theta^2+\mu_{33}\,\theta^3,\\
\mu(x_2,x_3)&=\mu_{11}\,\theta^1+\mu_{21}\,\theta^2+\mu_{31}\,\theta^3,\\
\mu(x_3,x_1)&=\mu_{12}\,\theta^1+\mu_{22}\,\theta^2+\mu_{32}\,\theta^3,
\endaligned
$$
so that $\mu$ gets identified with the $3 \times 3$ matrix,
\begin{equation}\label{matriz-mu}
\mu \ \leftrightarrow\ 
\begin{pmatrix}
\mu_{11}&\mu_{12}&\mu_{13}\\
\mu_{21}&\mu_{22}&\mu_{23} \\ 
\mu_{31} & \mu_{32} & \mu_{33}
\end{pmatrix}.
\end{equation}
Observe that $\d\lambda = 0$ is trivially satisfied for the $3$-dimensional
Heisenberg Lie algebra $\h$. Indeed, 
since $[x_1,x_3]=[x_2,x_3]=0$ and $[x_1,x_2]=x_3$,
it follows that,
$$
\lambda([x_1,x_2],x_3)+\lambda([x_2,x_3],x_1)
+\lambda([x_3,x_1],x_2)=\lambda(x_3,x_3)=0.
$$
On the other hand, a straightforward computation shows that
$\d\mu+e_{\varphi}(\lambda)=0$, if and only if, 
\begin{eqnarray}
\label{condicion1}(\varphi^1 \lambda)_{11}
+(\varphi^2 \lambda)_{12}-(\varphi^1 \lambda)_{33}+\mu_{32}=0,\\
\label{condicion2}(\varphi^1 \lambda)_{21}+(\varphi^2 \lambda)_{22}
-(\varphi^2 \lambda)_{33}-\mu_{31}=0,\\
\label{condicion3}(\varphi^1 \lambda)_{31}+(\varphi^2 \lambda)_{32}=0.
\end{eqnarray}
These equations must retain their form after acting on them 
with elements of the group
$G\subset\Aut{\h}\times\GL(\a\oplus\h^*)$
consisting of those $\Psi$'s
as in {\bf Cor. \ref{equivariancia de ecuaciones}}
that preserve the isomorphism class of $\h(\lambda,\mu,\varphi)$.
In particular, we look for group elements $(g,\sigma)\in G$ with $g\in\Aut(\h)$ 
and,
\begin{equation}\label{pues-sigma}
\sigma=
\begin{pmatrix}
h & 0 \\
T & k
\end{pmatrix}\in  \GL(\F^{r} \oplus \h^{*}),
\end{equation}
where, $h \in \GL(\F^r)$, $k \in \GL(\h^{*})$ and $T\in\Hom_{\F}(\F^r,\h^*)$.
For any $g\in\GL(\h)$, write, $g(x_j)=\sum_{i=1}^3\,g_{ij}\,x_i$.
It is easy to see that $g\in\Aut(\h)$, if and only if its matrix has the form,
\begin{equation}\label{forma automorf g}
g=\begin{pmatrix}
g_{11} & g_{12} & 0 \\
g_{21} & g_{22} & 0 \\
g_{31} & g_{32} & g_{33}
\end{pmatrix}\!, \quad \text{with,}\quad g_{33}=g_{11}g_{22}-g_{12}g_{21} \neq 0.
\end{equation}
Now recall that a necessary condition
to preserve the isomorphism class of $\h(\lambda,\mu,\varphi)$
is that the representations
$R$ and $R^\prime$ be in the same $G$-orbit. Since,
$$
R(x)=\begin{pmatrix} 0 & 0 \\ \varphi(x) & \ad^*(x) \end{pmatrix},
\quad\forall\,x\in\h,
$$
we must have,
$$
\sigma\,\circ\,R(g^{-1}(x))\,\circ\,\sigma^{-1}=R^\prime(x),\quad\forall\,x\in\h.
$$
Thus, $\varphi$ can be changed into $\varphi^\prime$ inside $R^\prime$, as follows:
\begin{equation}\label{iso coadjunta}
{\aligned
\varphi^\prime(x) & = k\,(\,\varphi(g^{-1}(x)) - \ad^*(g^{-1}(x))\,k^{-1}\,T\,)\,h^{-1}, \\
\ad^*(x) & = k\,\ad^*(g^{-1}(x))\,k^{-1}.
\endaligned}
\end{equation}
It is straightforward to see that a linear map $k \in \GL_{\F}(\h^{*})$ 
satisfies the second condition in \eqref{iso coadjunta} if and only if:
\begin{equation}\label{forma de k}
k=
\left(
\begin{matrix}
k_{33}\,g_{33}\,{\begin{pmatrix} g_{11} & g_{12} \\ g_{21} & g_{22}\end{pmatrix}}^{-1} 
& {\begin{matrix} k_{13} \\ k_{23} \end{matrix}} \\
\ \ \ \ \ \ {\begin{matrix} 0 & \ \ 0 \end{matrix}} & k_{33}
\end{matrix}
\right)
,\quad
g_{33}=g_{11}g_{22}-g_{12}g_{21}.
\end{equation}
Also recall from {\bf Cor. \ref{equivariancia de ecuaciones}}, 
how $\lambda$ and $\mu$ are transformed via,
$\Lambda(\,\cdot\,)\mapsto \Lambda^\prime(\,\cdot\,)=\Phi(\gamma)\Lambda -\d(*)$, 
where, $\gamma=(g,\sigma)\in G$, with $g\in\Aut\h$ and $\sigma$ as in \eqref{pues-sigma}:
$$
\left(\Phi(\gamma)
(\Lambda)\right)(\,\cdot\,,\,\cdot\,)=
\Phi(\gamma)
\begin{pmatrix}
\lambda(\,\cdot\,,\,\cdot\,)
\\ 
\mu(\,\cdot\,,\,\cdot\,)
\end{pmatrix}
=
\sigma
\begin{pmatrix}
\lambda(g^{-1}(\,\cdot\,),g^{-1}(\,\cdot\,)) \\ \mu(g^{-1}(\,\cdot\,),g^{-1}(\,\cdot\,))
\end{pmatrix}
$$
$$
=
\begin{pmatrix}
h(\lambda(g^{-1}(\,\cdot\,),g^{-1}(\,\cdot\,)))\\
k\,(\mu(g^{-1}(\,\cdot\,),g^{-1}(\,\cdot\,)))+ T\,\circ\,\lambda(g^{-1}(\,\cdot\,),g^{-1}(\,\cdot\,))
\end{pmatrix}\,.
$$
We shall now determine a representative set of canonical forms
for the maps $\lambda$, under the left action,
$\lambda\mapsto(g,h).\lambda$, with 
$(g,h)\in \Aut(\h)\times\GL(\F^r)$,
$$
((g,h).\lambda)(\,\cdot\,,\,\cdot\,)=h(\lambda(g^{-1}(\,\cdot\,),g^{-1}(\,\cdot\,))).
$$
In terms of the corresponding matrices, it is easy to see that,
$$
(g,h).\lambda\ \leftrightarrow\ \displaystyle{\frac{1}{\det g}}\,h\,\lambda\,g^t, 
\quad\text{and}\quad (g,k).\mu \ \leftrightarrow\ \displaystyle{\frac{1}{\det g}}\,k\,\mu\,g^t.
$$

\smallskip
\begin{claim}{\sl
There exists an isomorphism 
$\Psi: \h(\lambda,\mu,\varphi)\to\h(\lambda^{\prime},\mu^{\prime},\varphi^{\prime})$, 
bringing the matrix of $\lambda^{\prime}$ into the form,
\begin{equation}\label{matriz forma de h}
\begin{pmatrix}
\lambda^{\prime}_{11} & \lambda^{\prime}_{12} & \lambda^{\prime}_{13} \\
0 & \lambda^{\prime}_{22} & \lambda^{\prime}_{23} \\
0 & 0 & \lambda^{\prime}_{33}\\
\vdots & \vdots & \vdots \\
0 & 0 & \lambda^{\prime}_{r3}
\end{pmatrix}.
\end{equation}
}
\end{claim}
\begin{proof}
Let  $g\in\Aut(\h)$ be
as in \eqref{forma automorf g}.
Write $h(v_\ell)=\sum_{k=1}^r\,h_{k\ell}\,v_k$, for $h \in \GL(\F^r)$ and let
$w_j=\sum_{\ell=1}^r\,\lambda_{\ell j}v_\ell$, with $1\le j\le 3$ and $1 \leq \ell \leq r$.
The $j$-th column vector of the matrix $h\,\lambda$ is, 
$$
\sum_{k=1}^r\,(h\,\lambda)_{kj}\,v_k
=\sum_{k,\ell=1}^r\,h_{k\ell}\,\lambda_{\ell j}\,v_k
=h\left(\sum_{\ell=1}^r\,\lambda_{\ell j}\,v_{\ell}\right)=h(w_j).
$$
So, we may symbolically write the matrix $h\,\lambda$ in the form
$$
h \, \lambda=\left(\,\,
h(w_1)\,\mid\, h(w_2)\,\mid\, h(w_3)\,\,\right),
\quad\text{where,}\quad 
w_j=\sum_{\ell=1}^r\,\lambda_{\ell j}v_\ell.
$$
If $\{w_1,w_2\}$ is a linearly independent set, we may choose $h\in\GL(\F^r)$
in such a way that $h(w_1)=v_1$ and $h(w_2)=v_2$. That is,
$$
h\,\lambda=
\begin{pmatrix}
1 & 0 & \ast \\
0 & 1 & \ast \\
\vdots & \vdots & \vdots \\
0 & 0 & \ast
\end{pmatrix}.
$$
On the other hand, if $\{w_1,w_2\}$ is a
linearly dependent set, say $w_2=\alpha\,w_1$
with $w_1\ne 0$, we may choose $h\in\GL(\F^r)$
in such a way that,
$$
h\,\lambda=
\begin{pmatrix}
1 & \alpha & \ast \\
0 & 0 & \ast \\
\vdots & \vdots & \vdots \\
0 & 0 & \ast
\end{pmatrix}.
$$
If $w_1= 0$, but $w_2\ne 0$, choose $h\in\GL(\F^r)$ in such a way that,
$$
h\,\lambda=
\begin{pmatrix}
0 & 1 & \ast \\
0 & 0 & \ast \\
\vdots & \vdots & \vdots \\
0 & 0 & \ast
\end{pmatrix}.
$$
In any case, this analysis shows that there is a choice of
$h\in\GL(\F^r)$ that brings the matrix of 
$h\,\lambda$ into of the form \eqref{matriz forma de h}.
Therefore, we might as well assume, right from the start, 
that the matrix of $\lambda$ is given by \eqref{matriz forma de h}.
\end{proof}

\smallskip
\begin{claim}{\sl
There exists an isomorphism 
$\Psi :\h(\lambda,\mu,\varphi)\to\h(\lambda^{\prime},\mu^{\prime},\varphi^{\prime})$,
that brings the matrix of $\lambda^{\prime}$ into the form,
\begin{equation}\label{forma canonica general}
\begin{pmatrix}
\lambda^{\prime}_{11} & 0 & 0 \\
0 & \lambda^{\prime}_{22} & 0 \\
\vdots & \vdots & \vdots \\ 
0 & 0 & 0
\end{pmatrix},
\qquad \lambda^{\prime}_{11},\ \lambda^{\prime}_{22} \in \F.
\end{equation}
}
\end{claim}

\smallskip
\begin{proof}
Choose $g\in \Aut(\h)$ to be upper triangular, so as to have
a lower triangular $g^t$. Then,
$\lambda^\prime=(\det g)^{-1}\lambda\,g^t$, and,
$$
\aligned
{
\begin{pmatrix}
\lambda^{\prime}_{11} & \lambda^{\prime}_{12} & \lambda^{\prime}_{13} \\
0 & \lambda^{\prime}_{22} & \lambda^{\prime}_{23} \\
0 & 0 & \lambda^{\prime}_{33}\\
\vdots & \vdots & \vdots \\
0 & 0 & \lambda^{\prime}_{r3} 
\end{pmatrix}
}
& =
\displaystyle{\frac{1}{\det g}}\,
{\begin{pmatrix}
\lambda_{11} & \lambda_{12} & \lambda_{13}\\
0 & \lambda_{22} & \lambda_{23} \\
0 & 0 & \lambda_{33}\\
\vdots & \vdots & \vdots \\
0 & 0 & \lambda_{r3}
\end{pmatrix}
}{
\begin{pmatrix}
g_{11} & g_{21} & g_{31}\\
0 & g_{22} & g_{32} \\
0 & 0 & g_{33}
\end{pmatrix}
}
\\
& = \displaystyle{\frac{1}{\det g}}\,
{
\begin{pmatrix}
\lambda_{11}g_{11} & \lambda_{11}g_{21}+\lambda_{12}g_{22}  & \,\,\ast\,\, \\
0 & \lambda_{22}g_{22} &  \,\,\ast\,\, \\
0 & 0 & \ast \\
\vdots & \vdots & \vdots \\
0 & 0 & \ast
\end{pmatrix}.
}
\endaligned
$$
If $\lambda_{12}\ne 0$, we may choose $g_{22}$ so as to make 
$\lambda^{\prime}_{12}=0$. Thus, we may also assume from the start that,
$$
\lambda =
\begin{pmatrix}
{\lambda}_{11} & 0 & \,\,\lambda_{13}\, \\
0 & {\lambda}_{22} & \,\,\lambda_{23}\, \\
0 & 0 & {\lambda}_{33}\\
\vdots & \vdots & \vdots \\
0 & 0 & \lambda_{r3}
\end{pmatrix}
$$
Now use the additional freedom for transforming the
$2$-cocycle $\lambda$ by means of a coboundary 
as in {\bf Cor. \ref{equivariancia de ecuaciones}} and \textbf{Cor. \ref{criterio cohomologico}}; that is, $\lambda \mapsto \lambda^\prime= \lambda+\d\tau$
for a $1$-cochain so as to obtain $\d\tau:\h \times \h \to\F^r$.
In particular, $\tau$ can be chosen in such a way that,
$$
\d\tau = \begin{pmatrix}
0 & 0 & -\lambda_{13} \\
0 & 0 & -\lambda_{23} \\ 
0 & 0 & -\lambda_{33}\\
\vdots & \vdots & \vdots \\
0 & 0 & -\lambda_{r3}
\end{pmatrix}.
$$
Therefore, $\lambda^\prime=\lambda+d\tau$ takes the form 
\ref{forma canonica general} as claimed.
\end{proof} 
\begin{Lemma}\label{formas canonicas}{\sl
The Lie algebra $\h(\lambda,\mu,\varphi)$ is isomorphic to 
$\h(\lambda^{\prime},\mu^{\prime},\varphi^{\prime})$, where the matrix of 
$\lambda^{\prime}$ can achieve one, and only one, of the following
canonical forms:
$$
\begin{pmatrix}
1 & 0 & 0 \\
0 & 1 & 0 \\
0 & 0 & 1 \\
0 & 0 & 0 \\
\vdots & \vdots & \vdots \\
0 & 0 & 0
\end{pmatrix},
\qquad 
\begin{pmatrix}
1 & 0 & 0 \\
0 & 0 & 0 \\
\vdots & \vdots & \vdots \\
0 & 0 & 0
\end{pmatrix},
\qquad
\begin{pmatrix}
0 & 0 & 0 \\
0 & 0 & 0 \\
\vdots & \vdots & \vdots \\
0 & 0 & 0
\end{pmatrix}.
$$
}
\end{Lemma}
\begin{proof}
Once $\lambda^\prime$ is brought into the form \eqref{forma canonica general}
it can be further multiplied on the left by a diagonal invertible matrix $h$ and 
the ordered pair $(\lambda_{11},\lambda_{22})$, formed by the diagonal entries, 
can be assumed to be either $(1,1)$, $(1,0)$, $(0,1)$ or $(0,0)$. 
As a matter of fact, the cases $(1,0)$ and $(0,1)$ define the same isomorphism
class for a given $\h(\lambda,\mu,\varphi)$ because the transformation 
$\lambda\mapsto h \lambda g^t=\lambda^\prime$ (with $\det g=1$) yields,
$$
\begin{pmatrix}
0 & 1 & 0 & 0  \\
1 & 0 & 0 & 0 \\
0 & 0 & 1 & 0 \\
0  & 0 & 0 & \operatorname{Id}_{r-3,r-3}
\end{pmatrix}
\begin{pmatrix}
1 & 0 & 0 \\
0 & 0 & 0 \\
0 & 0 & 0 \\
\vdots & \vdots & \vdots \\
0 & 0 & 0
\end{pmatrix}
\begin{pmatrix}
\,\,0 & 1 & 0 \\
\!-1 & 0 & 0 \\
\,\,0 & 0 & 1
\end{pmatrix}=\begin{pmatrix}
0 & 0 & 0 \\
0 & 1 & 0 \\
0 & 0 & 0 \\
\vdots & \vdots & \vdots \\
0 & 0 & 0
\end{pmatrix}.
$$
Now, the matrix \eqref{forma canonica general} with diagonal
entries $(1,1,0)$ lies in the same isomorphism class of 
$\h(\lambda^\prime,\mu^\prime,\varphi^\prime)$
where $\lambda^\prime$ has diagonal entries $(1,1,1)$, as claimed in the statement.
In fact, we simply change $\lambda$ by a coboundary coming from 
a linear map $\tau^{\prime}:\h\to \F^r$, such that $\tau^{\prime}(x_3)=v_r$,
so that,
$$
\d\tau^{\prime} \ \leftrightarrow\ 
\begin{pmatrix}
0 & 0 & 0 \\
0 & 0 & 0 \\
0 & 0 & 1 \\
\vdots & \vdots & \vdots \\
0 & 0 & 0
\end{pmatrix}.
$$
Finally, we claim that the Lie algebras $\h(\lambda,\mu,\varphi)$
and $\h(\lambda^\prime,\mu^\prime,\varphi^\prime)$ defined by the 2-cocycles 
$\lambda$ and $\lambda^{\prime}$ whose matrices are 
given by the first and second canonical forms of the statement,
cannot be isomorphic. 
If they were isomorphic, there should be invertible maps (matrices)
$h \in \GL_r(\F)$ and $g \in \Aut(\h)$ and a linear map $\tau^{\prime \prime}:\h \to \F^r$ such that:
$$
\frac{1}{\operatorname{det}\, g}h \begin{pmatrix} \operatorname{Id}_{3 \times 3} \\ 0 \end{pmatrix}g^t+\begin{pmatrix}
0 & 0 & \tau^{\prime \prime}_{13}\\
0 & 0 & \tau^{\prime \prime}_{23}\\
\vdots & \vdots & \vdots \\
0 & 0 & \tau^{\prime \prime}_{r3}
\end{pmatrix}=
\begin{pmatrix}
1 & 0 & 0 \\
0 & 0 & 0 \\
\vdots & \vdots & \vdots \\
0 & 0 & 0
\end{pmatrix}
$$
This equations lead us to solve the following:
$$
\frac{1}{\operatorname{det}\, g}
\begin{pmatrix}
h_{11} & h_{12}\\
h_{21} & h_{22}\\
\vdots & \vdots \\
h_{r1} & h_{r2}
\end{pmatrix}
\begin{pmatrix}
g_{11} & g_{21}\\
g_{12} & g_{22}
\end{pmatrix}
=\begin{pmatrix}
1 & 0 & 0 \\
0 & 0 & 0 \\
\vdots & \vdots & \vdots \\
0 & 0 & 0
\end{pmatrix}
$$
Then, $h_{21}=h_{22}=\cdots =h_{r1}=h_{r2}=0$, $h_{11}=g_{22}$ and $h_{12}=-g_{21}$. Thus, we obtain that $h$ takes the matrix form:
$$
\begin{pmatrix}
g_{22} & -g_{21} & h_{13} & \cdots & h_{1r}\\
0 & 0 &  h_{23} & \cdots & h_{2r}\\
\vdots & \vdots & \vdots & \vdots & \vdots  \\
0 & 0 & h_{r3} & \cdots & h_{rr}
\end{pmatrix},
$$
which contradicts the fact that $h$ is invertible.
\end{proof}

\smallskip
To proceed with the classification of the isomorphism classes of
the given family of Lie algebras $\h(\lambda,\mu,\varphi)$,
we shall now find a representative set of canonical forms 
for the skew-symmetric bilinear maps $\mu:\h \times \h \to \h^{*}$
to be coupled to each of the canonical forms $\lambda$ just found. 
Recall that $\h(\lambda,\mu,\varphi)$ and $\h(\lambda,\mu^\prime,\varphi^\prime)$
are in the same isomorphism class whenever,
\begin{equation}\label{act sobre mu}
\mu^\prime(\,\cdot\,,\,\cdot\,) = 
k\,(\mu(g^{-1}(\,\cdot\,),g^{-1}(\,\cdot\,)))+T(\lambda(g^{-1}(\,\cdot\,),g^{-1}(\,\cdot\,)))),
\end{equation}
where $T:\F^r \to \h^{*}$ is a linear map and the pair $(g,k) \in \Aut(\h) \times \GL(\g^{*})$, 
satisfies the second condition in \eqref{iso coadjunta}.
In this case, we shall impose first the restriction that the 
pairs $(g,h) \in \Aut(\h) \times \GL(\F^{r})$, lie in the isotropy group of
the found canonical forms for $\lambda$. For those $g$'s we then determine
$k$ as in \eqref{forma de k} and choose an appropriate $T$ so as to obtain
the desired canonical forms for the matrix,
$$
\mu^\prime = \displaystyle{\frac{1}{\det g}\,(\,k\,\mu\,g^t + T\,\lambda\,g^t\,)}.
$$
Thus, we now proceed in a case by case fashion.

\medskip
{\bf Case 1: $\lambda = \left(\begin{smallmatrix}
1 & 0 & 0 \\
0 & 1 & 0 \\
0 & 0 & 1 \\
\vdots & \vdots & \vdots \\
0 & 0 & 0 \end{smallmatrix}\right)$.}
\smallskip

Choose $g$, $h$ and $k$ to be the identity maps.
If the entries of $\mu$ are $\mu_{ij}$, choose $T:\F^r\to\h^*$ so that
$$
\aligned
T(v_j)&=-\sum_{i=1}^3\,\mu_{ij}\,\theta^i,\quad 1\le j\le 3,\\
T(v_j)&=0, \quad j \geq 4.
\endaligned
$$ 
These choices make
$\h(\lambda,\mu,\varphi)$ isomorphic to $\h(\lambda,0,\varphi^\prime)$. 
The canonical form for $\mu$ is in this case $\mu=0$.
 
\medskip
{\bf Case 2: $\lambda = \left(\begin{smallmatrix}
1 & 0 & 0 \\
0 & 0 & 0 \\
\vdots & \vdots & \vdots \\
0 & 0 & 0 \end{smallmatrix}\right)$.}
\smallskip

This time choose a linear map $\nu:\h \to \h^{*}$ and change
$\mu$ into $\mu+\d\nu$, whose matrix 
takes the form,
$$
\begin{pmatrix}
\mu_{11} & \mu_{12} & \mu_{13}+\nu_{13} \\
\mu_{21} & \mu_{22} & \mu_{23}+\nu_{23}\\
\mu_{31} & \mu_{32} & \mu_{33}+\nu_{33}
\end{pmatrix}\!,
$$
and $\nu_{13},\nu_{23},\nu_{33} \in \F$ are arbitrary. 
Thus, we may choose $\nu_{j3}=-\mu_{j3}$, $1 \le j \le 3$, 
to let $\mu+\d\nu$ achieve the matrix form
$\left(\begin{smallmatrix} \mu_{11} & \mu_{12} & 0 \\ 
\mu_{21} & \mu_{22} & 0 \\
\mu_{31} & \mu_{32} & 0  \end{smallmatrix} \right)$. 
Thus, we might as well assume that $\mu$ is already of this form.
Finally, we may further transform this $\mu$ into $\mu+T \circ \lambda$
by choosing the linear map $T:\F^r\to\h^*$ in such a way that
$T(v_1)=-\mu_{11}\theta^1-\mu_{21}\theta^2-\mu_{31}\theta^3$, 
$T(v_2)=\cdots=T(v_r)=0$. 
By taking $g$, $h$ and $k$, to be the identity maps in their corresponding domains,
we conclude that $\h(\lambda,\mu,\varphi)$
is in the same isomorphism class as 
$\h(\lambda,\mu+T\circ\lambda,\varphi^\prime)$, where the matrix of 
$\mu+T\circ\lambda$ is,
$$
\begin{pmatrix}
\mu_{11} & \mu_{12} & 0 \\
\mu_{21} & \mu_{22} & 0 \\
\mu_{31} & \mu_{32} & 0 \\
\end{pmatrix}+
\begin{pmatrix}
-\mu_{11} & 0 & 0 \\
-\mu_{21} & 0 & 0 \\
-\mu_{31} & 0 & 0
\end{pmatrix}
\begin{pmatrix}
1 & 0 & 0 \\
0 & 0 & 0 \\
0 & 0 & 0
\end{pmatrix}=\begin{pmatrix}
0 & \mu_{12} & 0 \\
0 & \mu_{22} & 0 \\
0 & \mu_{32} & 0 \\
\end{pmatrix}.
$$
Therefore, 
we may always assume that,
$\mu\leftrightarrow
\left(\begin{smallmatrix}0 & \mu_{12} & 0 \\0 & \mu_{22} & 0 \\0 & \mu_{33} & 0 \end{smallmatrix} \right)$.
Now, let $(g,h)\in\Aut\h\times\GL(\F^r)$ be such that $\lambda=(\det g)^{-1}h\,\lambda\,g^t$, 
where,
$\lambda=\left(\begin{smallmatrix}1 & 0 & 0 \\0 & 0 & 0 \\\vdots & \vdots & \vdots \\ 0 & 0 & 0 \end{smallmatrix} \right)$.
It is easy to see that $g$ and $h$ must be given by,
$$
h=\begin{pmatrix}
g_{11}g_{22}^2 & h_{12} & h_{13} & \cdots  & h_{1r} \\
0 & h_{22} & h_{23} & \cdots & h_{2r} \\
\vdots & \vdots & \vdots & \cdots & \vdots \\
0 & h_{r2} & h_{r3} & \cdots & h_{rr}
\end{pmatrix}
\quad\text{and}\quad
g=\begin{pmatrix}
g_{11} & g_{21} & 0 \\
0 & g_{22} & 0 \\
0 & g_{32} & g_{33}
\end{pmatrix},
$$
respectively.
Moreover, since $(g,k)$ must fix the coadjoint representation of $\h$, it
follows that $k$ has the form (see \eqref{forma de k}),
$$
k=\begin{pmatrix}
g_{22}k_{33} & -g_{21}k_{33} & k_{13} \\
0 & g_{11}k_{33} & k_{23} \\
0 & 0 & k_{33}
\end{pmatrix}.
$$
Thus, 
$\mu\leftrightarrow
\left(\begin{smallmatrix}0 & \mu_{12} & 0 \\0 & \mu_{22} & 0 \\0 & \mu_{33} & 0 \end{smallmatrix} \right)$
may be transformed into,
$$
\frac{1}{\det g}k \,\mu\, g^t
=\frac{1}{\det g}\begin{pmatrix}
0 & \mu_{12}^{\prime} & 0 \\
0 & \mu_{22}^{\prime} & 0 \\
0 & \mu_{32}^{\prime} & 0
\end{pmatrix},
$$
and $\h(\lambda,\mu,\varphi)\simeq\h(\lambda,\mu^\prime,\varphi^\prime)$. Observe that,
\begin{equation}\label{clasificacion}
\begin{array}{ll}
\mu_{12}^{\prime}&=g_{22}(g_{22}k_{33}\mu_{12}-g_{21}k_{33}\mu_{22}+k_{13}\mu_{32}),\\
\mu_{22}^{\prime}&=g_{22}(g_{11}k_{33}\mu_{22}+k_{23}\mu_{32}),\\
\mu_{32}^{\prime}&=g_{22}k_{33}\mu_{32}.
\end{array}
\end{equation}
If $(\mu_{22},\mu_{32}) \neq (0,0)$, 
we may choose $g$ and $k$, so as to make $\mu_{12}^{\prime}=0$.
Therefore, we obtain {\it two\/} possible types of canonical forms
for $\mu$ when $\lambda=\left(\begin{smallmatrix}1 & 0 & 0 \\0 & 0 & 0 \\0 & 0 & 0 \\ \vdots & \vdots & \vdots \\ 0 & 0 & 0 \end{smallmatrix} \right)$;
namely,
$$
\aligned
\mu_1 & =
{\begin{pmatrix}
0 & \mu_{12} & 0\\
0 & 0 & 0 \\
0 & 0 & 0
\end{pmatrix}},
\qquad\text{or else,}
\\
\mu_2&=
{\begin{pmatrix}
0 & 0 & 0\\
0 & \mu_{22} & 0 \\
0 & \mu_{32} & 0
\end{pmatrix}
},
\qquad\text{with}\quad(\mu_{22},\mu_{32}) \neq (0,0).
\endaligned
$$
We claim that for the canonical form of $\lambda$ under consideration,
the Lie algebra extensions $\h(\lambda,\mu_1,\varphi_1)$ 
and $\h(\lambda,\mu_2,\varphi_2)$, can never be isomorphic,
regardless of the choices of $\varphi_1$ and $\varphi_2$.
Were they isomorphic, there would be $g$, $h$, $k$, $T$ and $\nu$,
with $(g,h).\lambda=\lambda$, $(g,k).\ad^*=\ad^*$,
an arbitrary map $T\in\Hom_{\F}(\F^r,\h^*)$,
and an appropriate map $\nu\in\Hom{\F}(\h,\h^*)$, satisfying,
$(g,k).\mu_1+(g,T).\lambda+d\nu=\mu_2$.
A straightforward computation, however, shows that
the matrix form of this equation is,
$$
\frac{1}{\det g}\begin{pmatrix}
g_{11}T_{11}+g_{21}g_{22}k_{33}\mu_{12} & g_{22}^2k_{33}\mu_{12} & \ast \\
g_{11}T_{21} & 0 & \ast \\
g_{11}T_{31} & 0 & \ast
\end{pmatrix}=
\begin{pmatrix}
0 & 0 & 0 \\
0 & \mu_{22} & 0 \\
0 & \mu_{32} & 0 \\
\end{pmatrix}\!,
$$
but this implies that $(\mu_{22},\mu_{32})=(0,0)$, contrary to our assumption.
Now, from \eqref{clasificacion} it is clear that 
$\mu_{32}=0$ if and only if ${\mu^\prime}_{32}=0$.
If $\mu_{32}\ne 0$ we may choose $k_{23}=-\mu_{32}^{-1}g_{11}k_{33}\mu_{22}$
to make ${\mu^\prime}_{22}=0$.
On the other hand, if  $\mu_{32}= 0$, we have
${\mu^\prime}_{22}=g_{11}g_{22}k_{33}\mu_{22}$.
In this case, $\mu_{22}\ne 0$, and 
one may choose $g_{11}g_{22}k_{33}$ so as to have 
${\mu^\prime}_{22}=1$. In summary,
we have found {\it four different\/} canonical forms 
for $\mu$; namely,
$$
\begin{pmatrix}
0 & 0 & 0 \\
0 & 0 & 0 \\
0 & 0 & 0
\end{pmatrix},
\quad
\begin{pmatrix}
0 & 1 & 0 \\
0 & 0 & 0 \\
0 & 0 & 0
\end{pmatrix},
\quad
\begin{pmatrix}
0 & 0 & 0 \\
0 & 1 & 0 \\
0 & 0 & 0
\end{pmatrix},
\quad
\begin{pmatrix}
0 & 0 & 0 \\
0 & 0 & 0 \\
0 & 1 & 0
\end{pmatrix}.
$$

{\bf Case 3: $\lambda = \left(\begin{smallmatrix}
0 & 0 & 0 \\
0 & 0 & 0 \\
\vdots & \vdots & \vdots \\
0 & 0 & 0 \end{smallmatrix}\right)$.}
\smallskip

In this case, the action given in \eqref{act sobre mu}, 
transforms $\mu$ into $(\det g)^{-1}\,k\,\mu\,g^t$, where the pair
$(g,k)\in\Aut(\h)\times\GL(\h^*)$ satisfies the second condition of
\eqref{iso coadjunta}.
Since $\lambda=0$, the condition $\d\mu+e_{\varphi}(\lambda)=0$
reduces to $\d\mu=0$ and from \eqref{condicion1}, \eqref{condicion2} and 
\eqref{condicion3}, we conclude that $\mu_{31}=\mu_{32}=0$. 

\smallskip
Let $g=\left(\begin{smallmatrix} A & 0 \\ 0 & g_{33} \end{smallmatrix}\right)\in\Aut(\h)$.
From \eqref{forma de k}, we get 
$k=k_{33}g_{33}\,\left(\begin{smallmatrix} (A^{-1})^t & \!\!0 \\0 & (g_{33})^{-1} \end{smallmatrix}\right)$.
By considering $\mu+\d\nu^{\prime}$, for $\nu^{\prime}:\h \to \h^{*}$, 
we might as well assume that right from the start $\mu$ has the form
$\left(\begin{smallmatrix} \mu_{11} & \mu_{12} & 0 \\\mu_{21} & \mu_{22} & 0 \\ 
0 & 0 & 0 \end{smallmatrix}\right)
=\left(\begin{smallmatrix} M & 0 \\ 0 & 0 \end{smallmatrix}\right)$ with the obvious
definition of the $2\times 2$ matrix $M$. Therefore, 
$$
\displaystyle{\frac{1}{\det g}}\,k\,\mu\,g^t
=
\displaystyle{\frac{k_{33}}{g_{33}}}
\begin{pmatrix}
(A^t)^{-1}MA^t & 0 \\
\,\,\,0 & 0
\end{pmatrix}\!.
$$
We may now choose $A$ in such a way as to 
bring $(A^t)^{-1}MA^t$ into its Jordan canonical form. 
This proves that if $\lambda=0$, there will be four 
possible canonical forms for $\mu$, each of which
defines an isomorphism class of its own; namely,
$$
\begin{pmatrix} \mu_{11}& 0 & 0 \\ 0 & \mu_{22} & 0 \\ 0 & 0 & 0\end{pmatrix};
\quad
\begin{pmatrix} \mu_{11}& 0 & 0 \\ \mu_{21} & \mu_{11} & 0 \\ 0 & 0 & 0\end{pmatrix};
\quad
\begin{pmatrix} \mu_{11}& 0 & 0 \\ 0 & \mu_{11} & 0 \\ 0 & 0 & 0 \end{pmatrix};
\quad
\begin{pmatrix}
0& 0 & 0 \\ 0 & 0 & 0 \\ 0 & 0 & 0 \end{pmatrix};
$$
$$
\!\!\!\!\!\!\!\!\!\!\!\!\!\!\!\!\!\!\!\!\!\!\!\!\!\!\ \ 
\mu_{11}\ne\mu_{22};
\ \qquad\qquad  \mu_{21}\ne 0;\ \qquad\qquad  \mu_{11}\ne 0.
\qquad\qquad 
$$
Finally, we can modify the first three of these canonical forms
by adding up a coboundary $d\nu$, changing the given $\mu$'s into
$\mu+\d\nu$'s so as to put in an additional $\mu_{11}$ diagonal entry
in their lower right corners. Then, by an additional rescaling, we
may further assume that $\mu_{11}=1$.
We summarize our findings in the following:

\smallskip
\begin{Prop}\label{clasif de algs}{\sl
Under the hypotheses of this section,
$\h(\lambda^{\prime},\mu^{\prime},\varphi^{\prime})$ is 
isomorphic to $\h(\lambda,\mu,\varphi)$
under one and only one of the following set of canonical forms
for the initial data $(\lambda,^{\prime}\mu^{\prime},\varphi^{\prime})$
\begin{itemize}
\item[{\bf 1.1}] 
$\lambda=
\left(\begin{smallmatrix} 1&0&0\\0&1&0\\0&0&1 \\ \vdots & \vdots & \vdots \\ 
0 & 0 & 0 \end{smallmatrix}\right)$,
$\mu=
\left(\begin{smallmatrix} 0&0&0\\0&0&0\\0&0&0\end{smallmatrix}\right)$,
and
$$
\aligned
{\varphi^{1}}_{11}+{\varphi^{2}}_{12}-{\varphi^{1}}_{33}&=0,\\
{\varphi^{1}}_{21}+{\varphi^{2}}_{22}-{\varphi^{2}}_{33}&=0,\\
{\varphi^{1}}_{31}+{\varphi^{2}}_{32}&=0.
\endaligned
$$
\smallskip

\item[{\bf 2.1.}]
$\lambda=
\left(\begin{smallmatrix} 1 & 0 & 0 \\ 0 & 0 & 0 \\ \vdots & \vdots & \vdots \\ 
0 & 0 & 0\end{smallmatrix}\right)$,
$\mu=
\left(\begin{smallmatrix} 0 & 0 & 0 \\ 0 & 0 & 0 \\ 0 & 0 & 0\end{smallmatrix}\right)$,
and 
$$
{\varphi^{1}}_{11}=0,\quad {\varphi^{1}}_{21}=0,\quad {\varphi^{1}}_{31}=0.
$$
\smallskip

\item[{\bf 2.2.}]
$\lambda=
\left(\begin{smallmatrix} 1 & 0 & 0 \\ 0 & 0 & 0 \\ \vdots & \vdots & \vdots \\ 
0 & 0 & 0\end{smallmatrix}\right)$,
$\mu=
\left(\begin{smallmatrix} 0 & 1 & 0 \\ 0 & 0 & 0 \\ 0 & 0 & 0\end{smallmatrix}\right)$,
and 
$$
{\varphi^{1}}_{11}=0,\quad {\varphi^{1}}_{21}=0,\quad {\varphi^{1}}_{31}=0.
$$
\smallskip

\item[{\bf 2.3.}]

$\lambda=
\left(\begin{smallmatrix} 1 & 0 & 0 \\ 0 & 0 & 0 \\ \vdots & \vdots & \vdots \\ 
0 & 0 & 0\end{smallmatrix}\right)$,
$\mu=
\left(\begin{smallmatrix} 0 & 0 & 0 \\ 0 & 1 & 0 \\ 0 & 0 & 0\end{smallmatrix}\right)$,
and 
$$
{\varphi^{1}}_{11}=0,\quad {\varphi^{1}}_{21}=0,\quad {\varphi^{1}}_{31}=0.
$$
\smallskip

\item[{\bf 2.4.}]
$\lambda=
\left(\begin{smallmatrix} 1 & 0 & 0 \\ 0 & 0 & 0 \\ \vdots & \vdots & \vdots \\ 
0 & 0 & 0\end{smallmatrix}\right)$,
$\mu=
\left(\begin{smallmatrix} 0 & 0 & 0 \\ 0 & 0 & 0 \\ 0 & 1 & 0\end{smallmatrix}\right)$,
and 
$$
{\varphi^{1}}_{11}+1=0,\quad {\varphi^{1}}_{21}=0,\quad {\varphi^{1}}_{31}=0.
$$
\smallskip

\item[{\bf 3.1.}] 
$\lambda=
\left(\begin{smallmatrix} 0 & 0 & 0 \\ \vdots & \vdots & \vdots \\ 0 & 0 & 0\end{smallmatrix}\right)$,
$\mu=\left(\begin{smallmatrix} 1 &  0 & 0 \\ 0 & \mu_{22}& 0 \\ 0 & 0 & 1\end{smallmatrix}\right)$,
$\mu_{22} \neq 1$,
with arbitray $\varphi^j$'s.
\medskip

\item[{\bf 3.2.}] $\lambda=
\left(\begin{smallmatrix} 0 & 0 & 0 \\ \vdots & \vdots & \vdots \\ 0 & 0 & 0\end{smallmatrix}\right)$,
$\mu=\left(\begin{smallmatrix} 1 &  0 & 0 \\ \mu_{21} & 1& 0 \\ 0 & 0 & 1\end{smallmatrix}\right)$, $\mu_{21} \neq 0$,
with arbitray $\varphi^j$'s.
\medskip

\item[{\bf 3.3.}]
$\lambda=
\left(\begin{smallmatrix} 0 & 0 & 0 \\ \vdots & \vdots & \vdots \\ 0 & 0 & 0\end{smallmatrix}\right)$,
$\mu=\left(\begin{smallmatrix} 1 &  0 & 0 \\ 0 & 1& 0 \\ 0 & 0 & 1\end{smallmatrix}\right)$
with arbitray $\varphi^j$'s.
\medskip

\item[{\bf 3.4.}]
$\lambda=
\left(\begin{smallmatrix} 0 & 0 & 0 \\ \vdots & \vdots & \vdots \\ 0 & 0 & 0\end{smallmatrix}\right)$,
$\mu=\left(\begin{smallmatrix} 0 &  0 & 0 \\ 0 & 0& 0 \\ 0 & 0 & 0\end{smallmatrix}\right)$
with arbitray $\varphi^j$'s.
\end{itemize}
}
\end{Prop}

\smallskip
\begin{proof}
What remains to be proved are the form of the matrix
entries of the $1$-cocycle $\varphi$, but these are given
by \eqref{condicion1}, \eqref{condicion2} and \eqref{condicion3}.
In each case, one only has to use the found canonical forms for $\lambda$ and $\mu$.
\end{proof}

\smallskip
Now, according to {\bf Thm. \ref{nilp cuad}}, 
a necessary condition for $\h(\lambda,\mu,\varphi)$ to admit an invariant metric,
is that the 2-cochain $\mu:\h \times \h \to \h^{*}$ satisfies the cyclic condition,
$\mu(x,y)(z)=\mu(y,z)(x)$, for all $x,y,z \in \h$. 
This ammounts to
have $\mu_{11}=\mu_{22}=\mu_{33}$. 
In particular, 
the Lie algebras given in {\bf 2.3} and {\bf 3.1}, of {\bf Prop. \ref{clasif de algs}},
cannot admit invariant metrics.
In addition, taking into account {\bf (b1)} and {\bf (b2)} of {\bf Thm. \ref{nilp cuad}}
for the relationship between $\lambda$ and $\varphi$ that has to be satisfied
when an invariant metric exists, we may now conclude the following:

\smallskip
\begin{Prop}\label{algs con metrica}{\sl
Under the hypotheses of this section,
$\h(\lambda,\mu,\varphi)$ 
admits an invariant metric if and only if,
\begin{itemize}
\item[{\bf 1.1.}] 
$\lambda=
\left(\begin{smallmatrix} 1&0&0\\0&1&0\\0&0&1 \\ 
\vdots & \vdots & \vdots \\ 0 & 0 & 0 \end{smallmatrix}\right)$,
$\mu=
\left(\begin{smallmatrix} 0&0&0\\0&0&0\\0&0&0\end{smallmatrix}\right)$,
and
$$
\varphi^1=\left(\begin{smallmatrix}
0 & 0 & 0 & \cdots & 0 \\
0 & 0 & -1  & \cdots & 0\\
0 & 1 & 0 & \cdots & 0
\end{smallmatrix}\right)\!,\quad 
\varphi^2=\left(\begin{smallmatrix}
\,\,0 & 0 & 1 & \cdots & 0 \\
\,\,0 & 0 & 0 & \cdots & 0 \\
\!-1 & 0 & 0 & \cdots & 0
\end{smallmatrix}\right)\!,\quad 
\varphi^3=\left(\begin{smallmatrix}
0 & \!-1 & 0 & \cdots & 0 \\
1 & \,\,0 & 0 & \cdots & 0 \\
0 & \,\,0 & 0 & \cdots & 0
\end{smallmatrix}\right)\!.
$$
\item[{\bf 2.1.}] 
$\lambda=
\left(\begin{smallmatrix} 1 & 0 & 0 \\ 0 & 0 & 0 \\ 0 & 0 & 0\end{smallmatrix}\right)$,
$\mu=
\left(\begin{smallmatrix} 0 & 0 & 0 \\ 0 & 0 & 0 \\ 0 & 0 & 0\end{smallmatrix}\right)$,
and 
$$
\varphi^1=0,\quad 
\varphi^2=\left(\begin{smallmatrix}
0 & 0 & -1 & \cdots & 0 \\
0 & 0 & 0 & \cdots & 0\\
0 & 0 & 0 & \cdots & 0
\end{smallmatrix}\right)\!,\quad 
\varphi^3=\left(\begin{smallmatrix}
0 & 1 & 0 & \cdots & 0\\
0 & 0 & 0 & \cdots & 0\\
0 & 0 & 0 & \cdots & 0
\end{smallmatrix}\right).
$$
\item[{\bf 3.3.}] 
$\lambda=
\left(\begin{smallmatrix} 0 & 0 & 0 \\ \vdots & \vdots & \vdots \\ 0 & 0 & 0\end{smallmatrix}\right)$,
$\mu=\left(\begin{smallmatrix} 1 &  0 & 0 \\ 0 & 1 & 0 \\ 0 & 0 & 1\end{smallmatrix}\right)$,
and $\varphi^j=0$ ($1\le j\le 3$).
\medskip

\item[{\bf 3.4.}] 
$\lambda=
\left(\begin{smallmatrix} 0 & 0 & 0 \\ \vdots & \vdots & \vdots \\ 0 & 0 & 0\end{smallmatrix}\right)$,
$\mu=\left(\begin{smallmatrix} 0 &  0 & 0 \\ 0 & 0 & 0 \\ 0 & 0 & 0\end{smallmatrix}\right)$,
and $\varphi^j=0$ ($1\le j\le 3$).

\end{itemize}
}
\end{Prop}

\smallskip
\begin{proof}
We only have to use the conditions given in {\bf Thm. \ref{nilp cuad}}.
In particular $\lambda$ determines $\varphi$. On the other hand,
$\mu$ has to satisfy the cyclic property
$\mu(x,y)(z)=\mu(y,z)(x)$, for all $x,y,z \in \h$. 
Since $\mu$ is skew-symmetric,
it follows that the matrix associated to $\mu$ has to be diagonal 
with its diagonal entries equal among themselves. This restriction
rules out the possibilities given in 
{\bf 2.2}, {\bf 2.3}, {\bf 2.4}, {\bf 3.1} and {\bf 3.2} 
of {\bf Prop. \ref{clasif de algs}}.
\end{proof}

\smallskip
Observe that the quadratic Lie algebra 
of case {\bf 3.4} in {\bf Prop. \ref{algs con metrica}}
is decomposable. 
In fact, $\g$ is the orthogonal direct sum of
$\h \oplus \h^{*}$ and $\F^r$.
Now, in order to exhibit at least one of the quadratic Lie algebras included 
in {\bf Prop. \ref{algs con metrica}}, we shall explicitly write 
the Lie brackets for case {\bf 1.1} of {\bf Prop. \ref{algs con metrica}}; namely,
$$
\aligned
& [x_1,x_2]=x_3+v_3,\quad [x_2,x_3]=v_1,\quad [x_3,x_1]=v_2,\\
& [x_1,v_2]=-\theta^3,\quad [x_1,v_3]=\theta^2,\\
& [x_2,v_1]=-\theta^3,\quad [x_2,v_3]=\theta^1,\\
& [x_3,v_1]=\theta^2,\quad [x_3,v_2]=-\theta^1,\\
& [x_1,\theta^3]=-\theta^2,\quad [x_2,\theta^3]=-\theta^1,\\
& [x,v_j]=0,\quad\text{for any}\ x \in \g,\quad  j \geq 4.
\endaligned
$$
Write $\F^3=\operatorname{Span}_{\F}\{v_1,v_2,v_3\}$ and
$\F^{r-3}=\operatorname{Span}_{\F}\{v_4,\cdots,v_r\}$,
so that $\F^r=\F^3 \oplus \F^{r-3}$.
Since $\operatorname{Im}(\lambda) \subset \F^3$ and $\F^{r-3} \subset C(\g)$, 
it follows that $\F^{r-3}$ is a non-degenerate central ideal of $\g$. 
Whence, $\g$ is the orthogonal direct sum of 
$\h \oplus \F^3 \oplus \h^{\ast}$ and $\F^{r-3}$, 
where the quadratic Lie algebra structure of 
$\h \oplus \F^3 \oplus \h^{\ast}$ corresponds to the one given 
by case {\bf 2.1} in {\bf Prop. \ref{algs con metrica}}, with $r=3$.
Observe that all the other quadratic Lie algebras 
given by {\bf Prop. \ref{algs con metrica}}, have 
$\Im(\lambda) \subset \F^{3}$ and therefore $\F^{r-3} \subset C(\g)$ is
a non-degenerate central ideal of $\g$. We then obtain the following result:

\smallskip
\begin{Prop}\label{z}{\sl
Let $\g$ be a finite dimensional solvable Lie algebra 
over an algebraically closed field $\F$ of characteristic zero
having one Abelian descending central ideal. 
Let $\mathfrak{i(\g)}$ and $\mathfrak{j}(\g)$ be the canonical ideals defined by,
$$
\mathfrak{i}(\g)=\sum_{k \in \N}C_k(\g) \cap \g^k,\quad \j(\g)=\bigcap_{k \in \N}(C_k(\g)+\g^k),
$$
where $C_k(\g)$ and $\g^k$ are the elements of the 
derived central series and descending central series, respectively. 
If $\g/\j(\g)$ is isomorphic to the 3-dimensional Heisenberg Lie algebra 
$\h$ and $\mathfrak{i}(\g) \neq \j(\g)$ with $r=\dim_{\F}\left(\j(\g)/\mathfrak{i}(\g) \right) \geq 3$, 
then there is a 3-dimensional vector space $\a$ and there is
an $r-3$-dimensional vector space $V$ such that $\g$ is isometric
to the orthogonal direct sum of $\h \oplus \a \oplus \h^{\ast}$ and $V$
and the quadratic Lie algebra structure on $\h \oplus \a \oplus \h^{\ast}$ 
corresponds to one of the cases given by {\bf Prop. \ref{algs con metrica}} for $r=3$.
}
\end{Prop}

\section*{Acknowledgements}

The three authors acknowledge the support received
through CONACyT Grant $\#$ A1-S-45886. 
Support by the special grant MB1411 is also acknowledged.
Also, RGD would like to thank the support 
provided by the post-doctoral fellowships
FOMIX-YUC 221183, FORDECYT 265667 and CONACYT 000153
that allowed him to stay at CIMAT - Unidad M\'erida.

\bibliographystyle{amsalpha}

\begin{thebibliography}{31}

\frenchspacing
  
\bibitem{BajoBenayadiMedina} 
Bajo, I., Benayadi, S., and Medina, A.
{\em Symplectic structures on quadratic Lie algebras.} 
Journal of Algebra
{\bf 316} (2007) 174-188.\hfill\break
https://doi.org/10.1016/j.jalgebra.2007.06.001

\bibitem{BenayadiSuper}
Benayadi, S.
{\em Socle and some invariants of quadratic Lie superalgebras.}
Journal of Algebra
{\bf 261} (2003) 245-291.\hfill\break
https://doi.org/10.1016/S0021-8693(02)00682-8

\bibitem{Bordemann} 
Bordemann,  M.,
{\em Nondegenerate invariant bilinear forms on nonassociative algebras.} 
Acta Math. Univ. Comenianae., Vol. LXVI, {\bf 2} (1997) 151-201.

\bibitem{Che} 
Chevalley, C., and Eilenberg, S.,
{\em Cohomology Theory of Lie Groups and Lie Algebras.} 
American Mathematical Society, $(1948)$.  


\bibitem{Kath}
Kath I., and Olbrich M., 
{\em Metric Lie algebras and quadratic extensions.}
Transformation Groups {\bf 11: 87} (2006) 
\hfill\break
https://doi.org/10.1007/s00031-005-1106-5

\bibitem{Med}
Medina, A., and Revoy, P.
{\em Alg\`ebres de Lie et produit scalaire invariant.}
Annales scientifiques de l'\`Ec. Norm. Sup. {\bf 4e}
s\'erie, tome {\bf 18} no 3 (1985) 553-561.\hfill\break
https://doi.org/10.24033/asens.1496

\bibitem{MCGS}
Rodr\'iguez-Vallarte, M.C., Salgado, G.,
{\em Geometric structures on Lie algebras and double extensions.}
Proc. Amer. Math. Soc. {\bf 146} (2018) 4199-4209\hfill\break
https://doi.org/10.1090/proc/14127 

\bibitem{Zhu}
Zhu, L. 
{\em Solvable quadratic Lie algebras.}
Science in China: Series A (2006) {\bf 49} 477-493\hfill\break
https://doi.org/10.1007/s11425-006-0477-y

\end{thebibliography}

\end{document}